\documentclass{amsart}
\usepackage{amsmath}
\usepackage{amssymb,amscd}
\usepackage{amsfonts}
\usepackage[enableskew]{youngtab}
\usepackage[enableskew]{youngtab}
\usepackage{graphicx}

\usepackage{epsfig,euscript,enumerate,cancel}

\newtheorem{theorem}{Theorem}[section]
\theoremstyle{plain}

\newtheorem{claim}{Claim}

\newtheorem{corollary}[theorem]{Corollary}

\newtheorem{definition}[theorem]{Definition}
\newtheorem{example}[theorem]{Example}

\newtheorem{lemma}[theorem]{Lemma}

\newtheorem{proposition}[theorem]{Proposition}
\newtheorem{remark}[theorem]{Remark}

\numberwithin{equation}{section}
\numberwithin{figure}{section}
\input{tcilatex}

\begin{document}
\title[Racah-Speiser Algorithm for Gromov-Witten Invariants]{A Combinatorial
Derivation of the Racah-Speiser Algorithm for Gromov-Witten invariants}
\author{Christian Korff }
\address{Department of Mathematics, University of Glasgow, Scotland, UK}
\email{c.korff@maths.gla.ac.uk}
\urladdr{http://www.maths.gla.ac.uk/\symbol{126}ck/}
\thanks{The author is financially supported by a University Research
Fellowship of the Royal Society.}
\date{18 October, 2009}
\begin{abstract}
Using a finite-dimensional Clifford algebra a new combinatorial product formula for the small quantum cohomology ring of the complex Grassmannian is presented. In particular, Gromov-Witten invariants can be expressed through certain elements in the Clifford algebra, this leads to a $q$-deformation of the Racah-Speiser algorithm allowing for their computation in terms of Kostka numbers. The second main result is a simple and explicit combinatorial formula for projecting product expansions in the quantum cohomology ring onto the $\mathfrak{\widehat{sl}}(n)$ Verlinde algebra. This projection is non-trivial and amounts to an identity between numbers of rational curves intersecting Schubert varieties and dimensions of moduli spaces of generalised $\theta$-functions.
\end{abstract}

\maketitle

% and its reduction to the $\widehat{\mathfrak{sl}}(n)_k$ Verlinde algebra}

%\subjclass[2000]{Primary 05C38, 15A15; Secondary 05A15, 15A18}
%\keywords{Keyword one, keyword two, keyword three}

\section{Introduction}
In representation theory the Racah-Speiser algorithm \cite{Racah}, \cite{Speiser} (see also \cite[Exercise 25.31]{FuHa}) computes multiplicities in the tensor product decomposition of irreducible modules of semi-simple Lie algebras in terms of weight multiplicities. For $\mathfrak{sl}(n)$ the tensor product multiplicities coincide with the celebrated Littlewood-Richardson coefficients and the weight multiplicities with the Kostka numbers, both of which can be defined combinatorially by counting Littlewood-Richardson and semi-stan\-dard tableaux, respectively ( see e.g. \cite{FultonYT} for details). Alternatively, one can interpret the Littlewood-Richardson coefficients as intersection numbers of Schubert varieties, i.e. as structure constants of the cohomology ring $H^{\ast }(\func{Gr}_{n,n+k})$ of the complex Grassmannian $\func{Gr}_{n,n+k}$ of $n$-dimensional subspaces in $\mathbb{C}^{n+k}$.

Denote by $qH^{\ast }(\func{Gr}_{n,n+k})$ the \emph{small} \emph{quantum cohomology ring} which is a particular $q$-deformation of $H^{\ast}(\func{Gr}_{n,n+k})$ and whose structure constants are 3 point genus 0 Gromov-Witten invariants; details will be given in the text.
Employing the Clifford algebra (or \emph{free fermion}) formulation of $qH^{\ast }(\func{Gr}_{n,n+k})$ given in \cite[Part II]{ckcs} a new combinatorial product formula for the quantum cohomology ring is presented. As a consequence one obtains a modified, `quantum version' of the Racah-Speiser algorithm which allows one to compute also Gromov-Witten invariants in terms of Kostka numbers. Several explicit examples are provided and comparison is made with alternative methods such as the
{\em rim-hook algorithm} of Bertram, Ciocan-Fontanine and Fulton \cite{BCF}.

In the second part of the paper we discuss the $\widehat{\mathfrak{sl}}%
(n)_{k}$ fusion ring of the Wess-Zumino-Novikov-Witten (WZNW) model, also known as the Verlinde algebra. This is a non-trivial quotient of the quantum cohomology
ring; see Theorem \ref{quotient} below and \cite{ckcs} for details. In this
article an alternative description of the quotient is presented by proving a simple combinatorial identity between the structure constants of both rings; see
\eqref{curious} in the text. This identity between Gromov-Witten invariants
and fusion coefficients (the structure constants of the fusion ring) amounts to equating the number of rational curves
intersecting Schubert varieties with the dimension of moduli spaces of
generalised $\theta $-functions. Whether a geometric interpretation of this
result exists is currently an open problem. Exploiting the identity between the structure constants of both rings one can project the `quantum Racah-Speiser
algorithm' from the quantum cohomology ring onto the $\widehat{\mathfrak{sl}}%
(n)_{k}$ fusion ring and compare the result with what is known as {\em Kac-Walton formula} for fusion coefficients \cite{Kac}, \cite{Walton}, \cite{GoodmanWenzl}. In contrast to this known extension of the Racah-Speiser algorithm which employs the \emph{affine} Weyl group, the present algorithm only uses the \emph{finite} Weyl group, i.e. the symmetric group, and the outer Dynkin diagram automorphism of $\widehat{\mathfrak{sl}}(n)$. Finally, we make contact with the combinatorial description of the fusion ring contained in \cite[Part I]{ckcs} by demonstrating that it constitutes yet another algorithm which is `dual' to the one obtained by projection from the quantum cohomology ring.

\subsection{Free fermion formulation of quantum cohomology} I shall summarize the main results deferring proofs to Section \ref{sec:fermion}. Fix $N=n+k$ in $\mathbb{N}$. Here $n\in\mathbb{Z}_{\geq 0}$ is the dimension and $k=N-n$ the co-dimension, both of which are allowed to vary in the interval $[0,N]$ in what follows. Recall the following bijection between 01-words and partitions: denote by%
\begin{equation}
W_{n,N}=\left\{ w=w_{1}w_{2}\cdots w_{N}~\left\vert
~|w|=\tsum_{i}w_{i}=n\right. ,\;w_{i}\in \{0,1\}\right\}
\end{equation}%
the set of 01-words of length $N$ which contain $n$ one-letters. Denote their
positions from \emph{right to left} by $\ell _{1}>...>\ell _{n}$ with $1\leq
\ell _{i}\leq N$ . Then one has the following bijection $\mathfrak{P}_{\leq
n,k}\rightarrow W_{n,N}$
\begin{equation}\label{part2word}
\lambda \mapsto w(\lambda )=0\cdots 0\underset{\ell _{n}}{1}0\cdots 0%
\underset{\ell _{1}}{1}0\cdots 0,\qquad \ell _{i}(\lambda )=\lambda
_{i}+n+1-i\;.
\end{equation}%
We shall denote the image of the inverse of this map by $\lambda (w)$. This
correspondence can be easily understood graphically: the Young diagram of
the partition $\lambda $ traces out a path in the $n\times k$ rectangle
which is encoded in $w$. Starting from the left bottom corner in the $n\times k$ rectangle go one box right for each letter $0$ and one box up for each letter 1; see Figure \ref{fig:fermions} for an example.

\begin{figure}[tbp]
\includegraphics[scale=0.3]{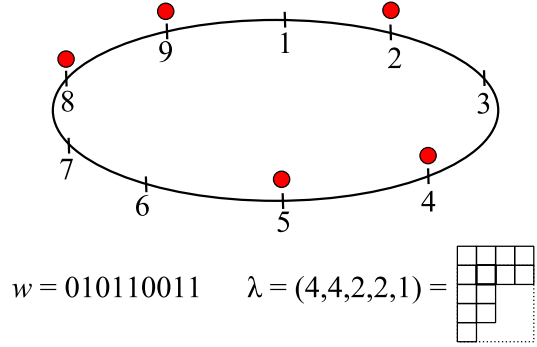}
\caption{Graphical depiction of a fermion configuration with $n=5$
particles, $k=4$ holes on a circle with $n+k=9$ sites. Below are the
corresponding 01-word $w$, partition $\protect\lambda=\protect\lambda(w)$
and its Young diagram.}
\label{fig:fermions}
\end{figure}

Consider the vector space%
\begin{equation}
\mathfrak{F}_{N}=\tbigoplus_{n=0}^{N}\mathfrak{F}_{n,N},\qquad \mathfrak{F}%
_{n,N}=\mathbb{C}W_{n,N},
\end{equation}%
where we set $\mathfrak{F}_{0}=\mathbb{C}\{00\cdots 0\}=\mathbb{C}$ and
refer to the zero-word $w=00\cdots 0$ as the\emph{\ vacuum vector} $%
\varnothing $. Define the integers $n_{i}(w)=w_{1}+\cdots +w_{i}$ which
count the number of $1$-letters lying in the closed interval $[1,i]$. For $%
1\leq i\leq N$ define the linear maps $\psi _{i}^{\ast },\psi _{i}:\mathfrak{%
F}_{n,N}\rightarrow \mathfrak{F}_{n\pm 1,N}$ \cite{ckcs},
\begin{eqnarray*}
\psi _{i}^{\ast }(w):= &&%
\begin{cases}
(-1)^{n_{i-1}(w)}w^{\prime }, & \text{$w_{i}=0$ and $w_{j}^{\prime
}=w_{j}+\delta _{i,j}$} \\
0, & \text{$w_{i}=1$}%
\end{cases}
\\
\psi _{i}(w):= &&%
\begin{cases}
(-1)^{n_{i-1}(w)}w^{\prime }, & \text{$w_{i}=1$ and $w_{j}^{\prime
}=w_{j}-\delta _{i,j}$} \\
0, & \text{$w_{i}=0$}.%
\end{cases}%
\end{eqnarray*}%

That is, up to a sign factor $\psi _{i}^{\ast }$ adds a 1-letter in $w$ at
position $i$. If that is not possible (since $w_{i}=1$) then it sends $w$ to
zero. Similarly, $\psi _{i}$ adds a zero letter at position $i$ if allowed.
In terms of partitions $\psi _{i}^{\ast }$ is the map which adds to a Young
diagram of a partition $\mu $ its top row (thereby increasing its height) and then
subtracts a boundary ribbon starting in the $(i-n)$-diagonal and ending in
the top row. Similarly, $\psi _{i}$ subtracts the top row of a Young diagram
and adds a boundary ribbon.

\begin{example}
{\rm To visualize the action of $\psi _{i}^{\ast }$ consider the special
case $n=k=4$ and $\mu =(4,3,3,1)$: $\psi _{3}^{\ast }\mu $ is depicted in
the figure below, where the entries in the diagram label the diagonals. The $%
(3-n)=-1$-diagonal determines the start of the boundary ribbon (the shaded
boxes) which has to be {\em subtracted}:
\begin{equation*}
\includegraphics[scale=0.25]{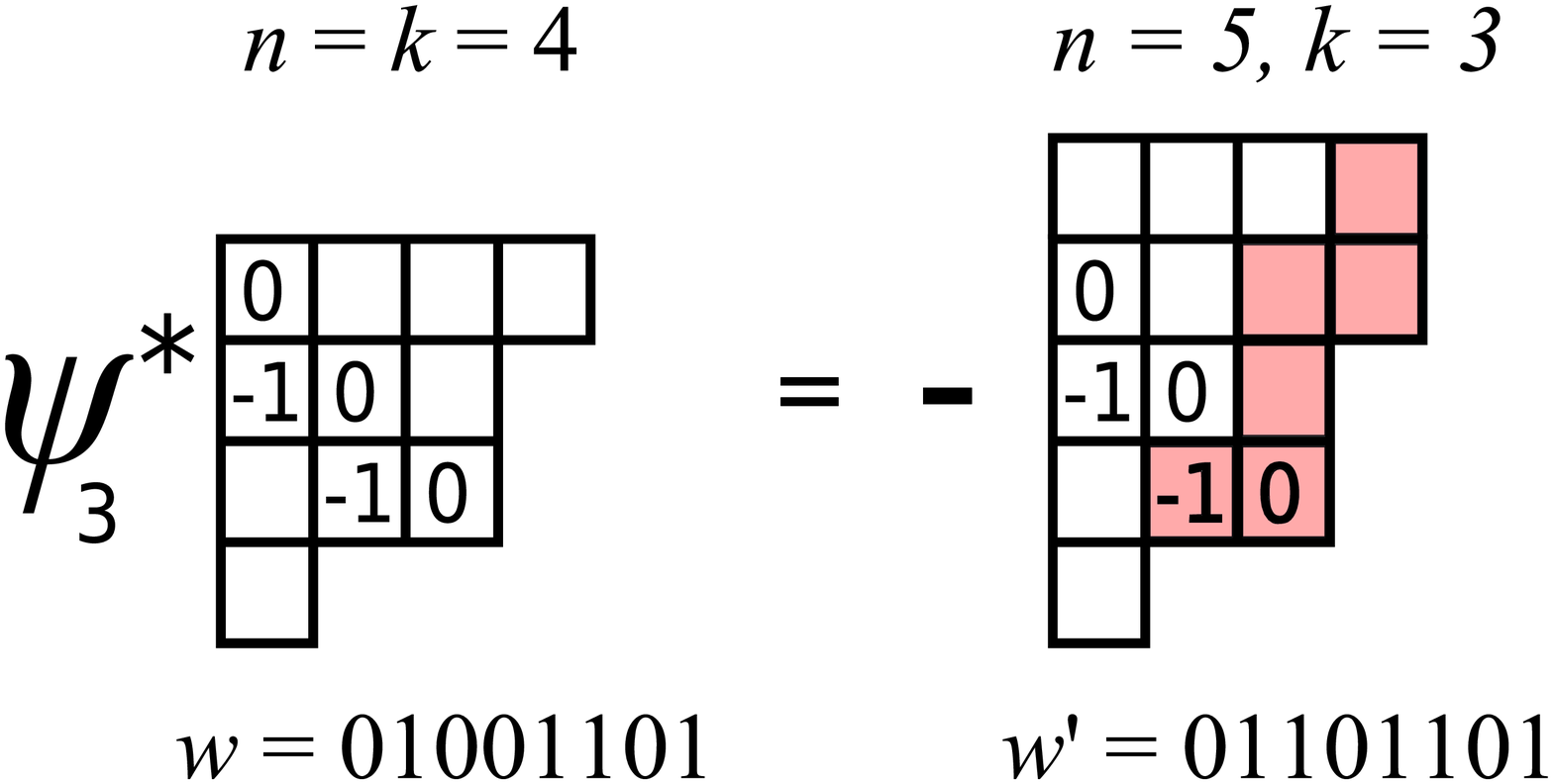}
\end{equation*}
For comparison, the action of $\psi_2$ on $\mu$ is
\begin{equation*}
\includegraphics[scale=0.25]{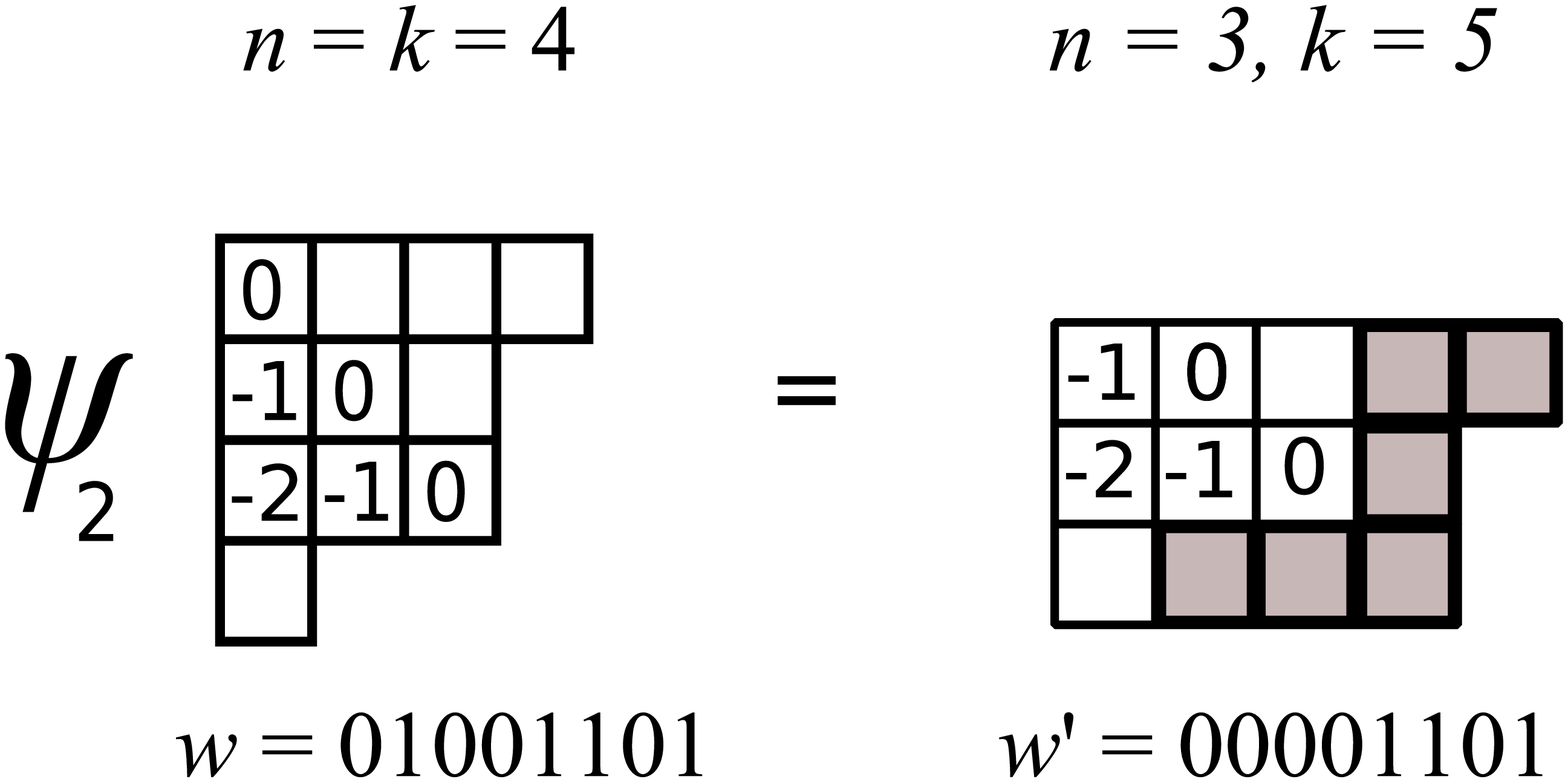}
\end{equation*}
where the shaded boxes now indicate the boundary ribbon which is {\em added} to obtain $\psi^\ast_2(\mu)=(5,4,4)$.}
\end{example}

The physical interpretation of these maps is the creation and annihilation
of a (quantum) particle at site $i$, respectively. Because we imposed the
Pauli exclusion principle, only one particle per site is allowed, we refer
to these particles as fermions. Henceforth, we shall interpret $\psi
_{i},\psi _{i}^{\ast }$ as elements in $\func{End}\mathfrak{F}_{N}$.

\begin{proposition}
The above endomorphisms $\psi _{i},\psi _{i}^{\ast }$ yield an irreducible
representation of the Clifford algebra with relations%
\begin{equation}
\psi _{i}\psi _{j}+\psi _{j}\psi _{i}=\psi _{i}^{\ast }\psi _{j}^{\ast
}+\psi _{j}^{\ast }\psi _{i}^{\ast }=0,\qquad \psi _{i}\psi _{j}^{\ast
}+\psi _{j}^{\ast }\psi _{i}=\delta _{ij}\;.  \label{Cliff relations}
\end{equation}%
Introducing the scalar product $\left\langle w,w^{\prime }\right\rangle
=\prod_{i}\delta _{w_{i},w_{i}^{\prime }}$ (anti-linear in the first factor)
one has the relation $\left\langle \psi _{i}^{\ast }w,w^{\prime
}\right\rangle =\left\langle w,\psi _{i}w^{\prime }\right\rangle $ for any
pair $w,w^{\prime }\in \mathfrak{F}_{N}$.
\end{proposition}
The proof is straightforward and can be found in \cite{ckcs}. We are now ready to state the fermion description of $qH^{\ast }(\func{Gr}_{n,N})$. Let $q$ be the deformation parameter of the quantum cohomology ring
and extend the state space as follows, $\mathfrak{F}_{N}[q]=\mathbb{C}%
[q]\otimes _{\mathbb{C}}\mathfrak{F}_{N}$. Given any pair $w,w'\in \mathfrak{F}_{n,N}[q]$ of 01-words let $\lambda ,\mu \in \mathfrak{P}%
_{\leq n,N-n}$ be the corresponding partitions under the bijection \eqref{part2word}. Define the following product $\star:\mathfrak{F}%
_{n,N}[q]\times \mathfrak{F}_{n,N}[q]\rightarrow \mathfrak{F}_{n,N}[q]$,%
\begin{equation}
\lambda\star\mu:=\sum_{T}\psi _{\ell _{n}(\mu )+t_{n}}^{\ast }%
\bar{\psi}_{\ell _{n-1}(\mu )+t_{n-1}}^{\ast }\psi _{\ell _{n-2}(\mu
)+t_{n-2}}^{\ast }\bar{\psi}_{\ell _{n-3}(\mu )+t_{n-3}}^{\ast }\cdots
~\varnothing ,  \label{fermi product}
\end{equation}%
where the sum runs over all semi-standard tableaux $T=T(\lambda )$ of shape $%
\lambda $ with $t_{i}$ being the number of entries $i$ in $T$. The indices
of the fermion creation operators in (\ref{fermi product}) can be greater
than $N$; we set $\bar{\psi}_{i}^{\ast }=\psi _{i}^{\ast }$ for $i=1,\ldots
,N$ and
\begin{equation}
\psi _{i+N}^{\ast }=(-1)^{\limfunc{n}+1}q\psi _{i}^{\ast },\qquad \bar{\psi}%
_{i+N}^{\ast }=(-1)^{\limfunc{n}}q\bar{\psi}_{i}^{\ast },  \label{qpbc2}
\end{equation}%
where $\limfunc{n}=\sum_{i=1}^{N}\psi _{i}^{\ast }\psi _{i}$ is the `particle
number operator'.

\begin{theorem}[Fermion presentation of quantum cohomology]
\label{main} \qquad

\begin{itemize}
\item[(i)] $(\mathfrak{F}_{n,N}[q],\star)$ is a commutative,
associative and unital algebra. %over $\mathbb{C}[q]$.

\item[(ii)] Restrict $(\mathfrak{F}_{n,N}[q],\star)$ to $%
\mathbb{Z}[q]$, then $(\mathfrak{F}_{n,N}[q],\star)\cong qH^{\ast }(\func{Gr}_{n,N}).$ In particular, the structure constants $C_{\lambda\mu}^{\nu,d}$ in the product expansion
\begin{equation*}%\label{GWexpansion}
\lambda\star\mu=\sum_d\sum_\nu q^d C_{\lambda\mu}^{\nu,d}\nu
\end{equation*}
are the Gromov-Witten invariants with $C_{\lambda \mu }^{\nu ,d}=0$ unless $|\lambda |+|\mu |-|\nu|=Nd$. The partitions $\nu$ correspond via \eqref{part2word} to the 01-words obtained by acting with the fermion creation operators in \eqref{fermi product} on the vacuum.
\end{itemize}
\end{theorem}

Denote by $K_{\lambda ,\alpha }$ the number of semi-standard Young tableaux
of shape $\lambda $ and weight $\alpha $, i.e. when $\alpha $ is a partition $K_{\lambda ,\alpha }$ is the Kostka number. Kostka numbers appear as multiplicities in the representation theory of the
symmetric group; see e.g. \cite{FultonYT}. Recall that $K_{\lambda ,\alpha}=K_{\lambda ,\alpha'}$ for any permutation $\alpha'$ of $\alpha $.

\begin{corollary}[Quantum Racah-Speiser Algorithm]
\label{maincor} Let $\lambda ,\mu ,\nu \in \mathfrak{P}_{\leq n,k}$. Given a
permutation $\pi \in S_{n}$ set
\begin{equation*}
\alpha _{i}(\pi )=(\ell _{i}(\nu )-\ell _{\pi (i)}(\mu ))\func{mod}N\geq 0%
\text{\quad and\quad }d(\pi )=\#\{i~|~\ell _{i}(\nu )-\ell _{\pi (i)}(\mu
)<0\}\;.
\end{equation*}%
Then one has the following identity for Gromov-Witten invariants,%
\begin{equation}
C_{\lambda \mu }^{\nu ,d}=\sum_{\substack{ \pi \in S_{n}  \\ d(\pi )=d}}%
(-1)^{\ell (\pi )+(n-1)d}K_{\lambda ,\alpha (\pi )}\;.  \label{GWKostka}
\end{equation}%
\end{corollary}

Setting $q=0$ only the structure constants with $d=0$ survive and
the formula (\ref{GWKostka}) specialises to the following expression for
Littlewood-Richardson coefficients,%
\begin{equation}
c_{\lambda \mu }^{\nu }=\sum_{\substack{ \pi \in S_{n} \\ \ell _{i}(\nu
)-\ell _{\pi (i)}(\mu )\geq 0}}(-1)^{\ell(\pi) }K_{\lambda ,(\ell _{1}(\nu )-\ell
_{\pi (1)}(\mu ),\ldots ,\ell _{n}(\nu )-\ell _{\pi (n)}(\mu ))}\;.
\label{LRKostka}
\end{equation}%
Recall that the Kostka number $K_{\lambda ,\mu }$ gives the multiplicity of
the weight $\mu $ in the $\mathfrak{sl}(n)$-representation $V(\lambda )$ of
highest weight $\lambda $, while the Littlewood-Richardson coefficients
coincide with the multiplicity of the highest weight representation $V(\nu )$
in the tensor product decomposition of $V(\lambda )\otimes V(\mu )$. Thus,
this result can be interpreted as a combinatorial derivation of the \emph{%
Racah-Speiser algorithm} (also known as \emph{Weyl's method of characters})
and we shall therefore refer to \eqref{GWKostka} as `quantum Racah-Speiser algorithm'.\\

In light of the identity \eqref{GWKostka} recall that Kostka numbers can be computed either by Konstant's multiplicity formula (see e.g. \cite{FuHa}) or recursively (see e.g. \cite{MacDonald}),
\begin{equation}
K_{\lambda,(\alpha_1,\ldots,\alpha_{\ell},0,\ldots)}=\sum_\mu K_{\lambda-\mu,(\alpha_1,\ldots,\alpha_{\ell-1},0,\ldots)},
\end{equation}
where the sum runs over all compositions $\mu$ with $\mu_i\in\mathbb{Z}_{\geq 0}$ such that $\sum_i\mu_i=\alpha_\ell$ and $\lambda_i-\mu_i\geq\lambda_{i+1}$. The analogous result in representation theory is known as {\em Freudenthal recursion formula}.
\subsection{Projection onto the $\widehat{\mathfrak{sl}}(n)_{k}$ Verlinde
algebra}
The small quantum cohomology ring can be identified with the fusion ring or Verlinde algebra of the $\widehat{\mathfrak{u}}(n)\cong \widehat{\mathfrak{gl}}(n)$ WZNW model (at level $(k,n)$) when specialising to $q=1$; see \cite{Witten} and references therein. The latter model is a \emph{topological} field theory. Here we are interested in the Verlinde algebra of
the $\widehat{\mathfrak{su}}(n)_{k}\cong \widehat{\mathfrak{sl}}(n)_{k}$
WZNW model, which is a \emph{conformal} field theory. Because one has the
decomposition $\widehat{\mathfrak{u}}(n)\cong \widehat{\mathfrak{su}}%
(n)\otimes \widehat{\mathfrak{u}}(1)$ one also expects in this case a close
relationship, albeit somewhat less trivial, between $qH^{\ast }(\limfunc{Gr}%
_{n,n+k})$ and the $\widehat{\mathfrak{sl}}(n)_{k}$\ Verlinde algebra. In \cite%
{ckcs} this relationship has been made precise (see Theorem \ref{quotient}
below) using an explicit combinatorial description of both rings in terms of
noncommutative Schur polynomials. Here we shall give an alternative formulation in terms of the structure constants of both rings using a simple combinatorial recipe. Again, only the main results are summarized here, the proofs will be presented in Section \ref{sec:projection}.\\
%\smallskip

Recall the definition of the $\widehat{\mathfrak{sl}}(n)_{k}$ Verlinde
algebra: denote by%
\begin{equation}
P_{k}^{+}=\left\{ \hat{\lambda}=\sum_{i=0}^{n-1}m_{i}\hat{\omega}%
_{i}\;\left\vert \;\sum_{i=0}^{n-1}m_{i}=k\right. \!,\;m_{i}\in \mathbb{Z}%
_{\geq 0}\right\}  \label{domweights}
\end{equation}%
the set of all \emph{dominant integral weights of level} $k$, where the $%
\hat{\omega}_{i}$'s denote the fundamental affine weights of the affine Lie algebra $\widehat{\mathfrak{sl}}(n)$; see \cite{Kac} for details. Consider the free abelian group (with respect to
addition) generated by $P_{k}^{+}$ and introduce the so-called fusion product%
\begin{equation}
\hat{\lambda}\ast \hat{\mu}=\sum_{\hat{\nu}\in P_{k}^{+}}\mathcal{N}_{\hat{%
\lambda}\hat{\mu}}^{(k)\hat{\nu}}\hat{\nu},  \label{fusiondef}
\end{equation}%
where the structure constants $\mathcal{N}_{\hat{\lambda}\hat{\mu}}^{(k)\hat{%
\nu}}\in \mathbb{Z}_{\geq 0}$, called \emph{fusion coefficients}, can be
explicitly computed from the Verlinde formula \cite{Verlinde} (see equation %
\eqref{Verlinde} below) and equal the dimension of certain moduli spaces of
generalised $\theta $-functions; see \cite{Beauville} and references therein
for details. We shall denote the resulting unital, commutative and
associative algebra by $V_{k}(\widehat{\mathfrak{sl}}(n);\mathbb{C)}$ and
refer to $V_{k}(\widehat{\mathfrak{sl}}(n);\mathbb{Z)}$ as the \emph{fusion
ring}.\\

Product expansions in the quantum cohomology ring $qH^{\ast }(\limfunc{Gr}_{n,n+k})$ can be projected on to product expansions in the fusion ring. For this one needs to identify basis elements in both rings which requires the notion of column and row reduction of
partitions: for \emph{any} $n,k\in\mathbb{N}$ introduce the following two
maps
\begin{equation}
^{\prime }:\mathfrak{P}_{\leq n,k}\rightarrow \mathfrak{P}_{\leq
n-1,k}~,\;\;\lambda \mapsto \lambda ^{\prime }\quad \text{and}\quad ^{\prime
\prime }:\mathfrak{P}_{\leq n,k}\rightarrow \mathfrak{P}_{\leq
n,k-1}~,\;\;\lambda \mapsto \lambda ^{\prime \prime },  \label{partred}
\end{equation}%
where $\lambda ^{\prime }$ is the partition obtained by removing all columns
of maximal height (here $n$) from the Young diagram of $\lambda $ and $%
\lambda ^{\prime\prime }$ is the partition obtained after deleting all rows
of maximal length (here $k$). Obviously, we have $(\lambda ^{\prime
})^{t}=(\lambda ^{t})^{\prime \prime }$.

Furthermore, we observe that the set $P_{k}^{+}$ is in one-to-one
correspondence with the partitions $\mathfrak{P}_{\leq n-1,k}$ whose Young
diagram fits into the $(n-1)\times k$ bounding box. Namely, one defines a
bijection $P_{k}^{+}\longrightarrow \mathfrak{P}_{\leq n-1,k}$ by setting
\begin{equation}
\hat{\lambda}\longmapsto \lambda =(\lambda _{1},\ldots ,\lambda
_{n-1},0,0,\ldots )\quad \text{with}\quad \lambda _{i}-\lambda _{i+1}=m_{i}~,
\label{weight2part}
\end{equation}%
where $m_i$ is the so-called Dynkin label, i.e. the coefficient of the $i^{\text{th}}$ fundamental weight in \eqref{domweights}. Vice versa, given a partition $\lambda \in \mathfrak{P}_{\leq n-1,k}$ we
shall denote by $\hat{\lambda}$ the corresponding affine weight in $%
P_{k}^{+} $.

\begin{proposition}[Projection of Gromov-Witten invariants]
\label{mainprop}Let $\lambda ,\mu ,\nu \in \mathfrak{P}_{\leq n,k}$ and
denote by $\hat{\lambda}^{\prime },\hat{\mu}^{\prime },\hat{\nu}^{\prime
}\in P_{k}^{+}$ the inverse images of $\lambda ^{\prime },\mu ^{\prime },\nu
^{\prime }\in \mathfrak{P}_{\leq n-1,k}$ under the bijection (\ref%
{weight2part}). Then one has the following identity between the associated
Gromov-Witten invariant and fusion coefficient%
\begin{equation}
C_{\lambda \mu }^{\nu ,d}=\mathcal{N}_{\hat{\lambda}^{\prime }\hat{\mu}%
^{\prime }}^{(k),\limfunc{rot}^{d}(\hat{\nu}^{\prime })}\;,  \label{curious}
\end{equation}%
where $\func{rot}:P_{k}^{+}\rightarrow P_{k}^{+}$ is the $\widehat{\mathfrak{%
sl}}(n)$-Dynkin diagram automorphism of order $n$. Employing the second map in (\ref{partred}) the analogous equality holds for the $\widehat{%
\mathfrak{sl}}(k)_{n}$ fusion coefficient.
\end{proposition}

Thus, according to formula (\ref{curious}) we can compute from the product expansion of $\lambda\star\mu$ in the quantum cohomology ring the product expansion of $\hat\lambda'\ast\hat\mu'$ in the fusion ring by simply deleting columns of height $n$ in the associated Young diagrams and then `rotating' each term in the expansion with the Dynkin diagram automorphism $\func{rot}$; see Example \ref{projectionex} in the text. Note that this is a genuine projection as products of partitions in $qH^{\ast }(\limfunc{Gr}%
_{n,n+k})$ which differ only by $n$-columns are mapped onto the same products in the fusion ring. In fact, $\func{dim} qH^{\ast }(\limfunc{Gr}_{n,n+k})={n+k \choose n}$, while  $\func{dim} V_{k}(\widehat{\mathfrak{sl}}(n);\mathbb{Z})={n+k-1\choose n-1}$. Moreover, the identity \eqref{curious} implies via (\ref{fermi product})
and (\ref{GWKostka}) a modified Fermion product formula for $V_{k}(\widehat{%
\mathfrak{sl}}(n);\mathbb{Z)}$ and an expression for the fusion coefficients
in terms of Kostka numbers.\\
%, which we again can compare against other known
%expressions in the literature, such as the Kac-Walton formula \cite{Kac},
%\cite{Walton}. We find that the equality for the fusion coefficients
%resulting from (\ref{GWKostka}) and (\ref{curious}) constitutes a
%simplification as it directly expresses $\mathcal{N}_{\hat{\lambda}\hat{\mu}%
%}^{(k)\hat{\nu}}$ in terms of weight multiplicities (Kostka numbers) and
%thus circumvents the Littlewood-Richardson decomposition of tensor products
%and the use of the affine Weyl group implicit in the Kac-Walton formula.\\
\medskip

The article is organized as follows: for the benefit of the reader Section \ref{sec:reminder} reviews the definition of the quantum cohomology ring and states the precise relationship with the Verlinde algebra presented as a quotient in the ring of symmetric functions. Section \ref{sec:fermion} contains the proof of the new product formula \eqref{fermi product}, i.e. Theorem \ref{main}. We discuss explicit examples where we compare with the rim hook and other known algorithms in the literature. In Section \ref{sec:projection} we derive the projection formula \eqref{curious} using the Bertram-Vafa-Intrilligator and Verlinde formula for Gromov-Witten invariants and fusion coefficients, respectively. We also discuss how product expansions in the fusion ring can be `lifted' to the quantum cohomology ring. First `lifting' and then projecting we demonstrate how recursion formulae for Gromov-Witten invariants derived in \cite{ckcs} lead to analogous relations for the recursive computation of fusion coefficients. Again explicit examples are presented to illustrate the general formulae.\\
\medskip

{\bf Acknowledgement}: The author would like to thank Alastair Craw for many helpful discussions and Catharina Stroppel for a previous collaboration.

\section{Reminder: Quantum Cohomology and Gromov-Witten Invariants}\label{sec:reminder}
  Starting with the non-deformed cohomology ring we briefly recall the definition of the quantum cohomology ring; for details and references see e.g. \cite{Bertram}, \cite{Buch}, \cite{Tamvakis}.
%Choose two positive integers $n$ and $k$, and set $N=n+k$.
%Denote by $\func{%
%Gr}_{n,N}\equiv \limfunc{Gr}(n,\mathbb{C}^{N})$ the Grassmannian of $n$%
%-dimensional subspaces $V\subset \mathbb{C}^{N}$ and let $H^{\ast }(\func{Gr}%
%_{n,N})=H^{\ast }(\func{Gr}_{n,N};\mathbb{Z})$ be its cohomology ring.
Fix a standard flag, $\mathbb{C}^{1}\subset \mathbb{C}^{2}\subset \cdots \subset
\mathbb{C}^{N}$, then a basis of $H^{\ast }(\func{Gr}_{n,N})$ is given in terms of Schubert
classes $[\Omega _{\lambda }]$ which are the fundamental cohomology classes
of the Schubert varieties
\begin{equation}
\Omega _{\lambda }=\left\{ V\in \func{Gr}_{n,N}~|~\dim (V\cap \mathbb{C}%
^{k+i-\lambda _{i}})\geq i,\;i=1,2,\ldots n\right\} ,
\end{equation}%
where $\lambda =(\lambda _{1}\geq \lambda _{2}\geq \cdots \geq \lambda _{n})$
is a partition whose associated Young diagram fits into a $n\times k$
rectangle with $k=N-n$ as before. We shall identify partitions with their Young diagrams and denote
this set by $\mathfrak{P}_{\leq n,k}$. Within the basis of Schubert classes
the multiplication in $H^{\ast }(\func{Gr}_{n,N})$ is determined through the
product expansion%
\begin{equation}
\lbrack \Omega _{\lambda }]\cup \lbrack \Omega _{\mu }]=\sum_{\nu \in
\mathfrak{P}_{\leq n,k}}c_{\lambda \mu }^{\nu }[\Omega _{\nu }],\qquad
c_{\lambda \mu }^{\nu }:=\langle \Omega _{\lambda },\Omega _{\mu },\Omega
_{\nu ^{\vee }}\rangle ,
\end{equation}%
where the structure constants are the intersection numbers of the
corresponding Schubert varieties $\Omega _{\lambda },\Omega _{\mu },\Omega
_{\nu ^{\vee }}$ and $\nu ^{\vee }$ denotes the complement $(k-\nu
_{n},\ldots ,k-\nu _{1})~$of $\nu $ in the $n\times k$ rectangle. The
non-negative integers $c_{\lambda \mu }^{\nu }$ coincide with the celebrated
\emph{Littlewood-Richardson coefficients}. In particular, the map $[\Omega
_{\lambda }]\mapsto s_{\lambda}$, where $s_\lambda$ is the Schur polynomial in the ring of symmetric functions, provides a ring isomorphism%
\begin{equation}
H^{\ast }(\func{Gr}_{n,N})\cong \mathbb{Z}[e_{1},\ldots ,e_{n}]/\langle
h_{k+1},\ldots ,h_{n+k-1}\rangle
\end{equation}%
with $e_r=s_{(1^r)}$ and $h_r=s_{(r)}$ denoting the elementary and complete symmetric polynomial of degree $r$; see e.g. \cite{MacDonald} for the definitions of the mentioned symmetric functions. The quotient condition
$h_{k+1}=\cdots =h_{n+k-1}=0$ ensures that $s_{\nu }=0$ if $\nu \notin
\mathfrak{P}_{\leq n,k}$. Further details can be found in e.g. \cite%
{FultonYT}.

The (\emph{small}) \emph{quantum cohomology ring }$qH^{\ast }(\func{Gr}%
_{n,N})$ is isomorphic to $\mathbb{Z}[q]\otimes _{\mathbb{Z}}H^{\ast }(%
\limfunc{Gr}\nolimits_{n,k})$ as a $\mathbb{Z}[q]$-module, where $q$ is a
variable of degree $N=n+k$. Set $\sigma _{\lambda }=1\otimes \lbrack \Omega
_{\lambda }]$ and define the ring structure now with respect to the `$q$%
-deformed' product%
\begin{equation}
\sigma _{\lambda }\star \sigma _{\mu }=\sum_{d}\sum_{\nu \in \mathfrak{P}%
_{\leq n,k}}q^{d}C_{\lambda \mu }^{\nu ,d}\sigma _{\nu },  \label{*product}
\end{equation}%
where $C_{\lambda \mu }^{\nu ,d}=\left\langle \Omega _{\lambda },\Omega
_{\mu },\Omega _{\nu ^{\vee }}\right\rangle _{d}$ are the three-point, genus
zero Gromov-Witten invariants which count the number of rational curves of
finite degree $d$ intersecting generic translates of the Schubert varieties
specified by the partitions $\lambda ,\mu $ and $\nu ^{\vee }$. One can show
that $C_{\lambda \mu }^{\nu ,d}=0$ unless $|\lambda |+|\mu |-|\nu |=dN$. As
in the non-deformed case there exists also here a presentation in the ring
of symmetric functions which is due to Siebert and Tian \cite{ST},%
\begin{equation}
qH^{\ast }(\func{Gr}_{n,N})\cong \left( \mathbb{Z}[q]\otimes _{\mathbb{Z}}%
\mathbb{Z}[e_{1},\ldots ,e_{n}]\right) /\langle h_{k+1},\ldots
,h_{n+k-1},h_{n+k}+(-1)^{n}q\rangle \;.  \label{STiso}
\end{equation}%
Again one identifies $\sigma _{\lambda }\mapsto s_{\lambda }$ under this
isomorphism. Setting $q=0$ one recovers the non-deformed cohomology ring $%
H^{\ast }(\func{Gr}_{n,N})$, i.e. the Gromov-Witten invariants specialise
for $d=0$ to the intersection numbers of the respective Schubert varieties.

There is a close relationship between the quantum cohomology ring and the $\widehat{\mathfrak{sl}}(n)_k$-Verlinde algebra which can be stated as follows:
\begin{theorem}%[Korff-Stroppel \cite{ckcs}]
\label{quotient} Employing the presentation \eqref{STiso} of the quantum
cohomology ring and $qH^{\ast }(\func{Gr}_{n,N})\cong qH^{\ast }(\func{Gr}_{k,N})$  one has the ring isomorphisms \cite[Theorem 1.3]{ckcs},
\begin{eqnarray}
V_{k}(\widehat{\mathfrak{sl}}(n);\mathbb{Z})&\cong& qH^{\ast }(\limfunc{Gr}%
\nolimits_{n,n+k})/\langle e_{n}-1,q-h_{k}\rangle\label{VQiso}\\
&\cong& qH^{\ast }(\limfunc{Gr}%
\nolimits_{k,n+k})/\langle h_{n}-1,q-e_{k}\rangle  \;.  \label{VQisodual}
\end{eqnarray}
%Alternatively, exploiting one has
%\begin{equation}
%V_{k}(\widehat{\mathfrak{sl}}(n);\mathbb{Z)\cong }qH^{\ast }(\limfunc{Gr}%
%\nolimits_{k,n+k})/\langle h_{n}-1,q-e_{k}\rangle \;.  \label{VQiso}
%\end{equation}
\end{theorem}
 In Section \ref{sec:projection} we will relate these isomorphisms to concrete algorithms for the computation of fusion coefficients. While the isomorphism \eqref{VQiso} has not explicitly been stated in \cite{ckcs} it is implicit in the results therein; we will discuss its proof also in Section \ref{sec:projection}.

\section{Proof and example of the fermion product formula}\label{sec:fermion}
The proof of Theorem \ref{main} is straightforward, however, it requires several known results which are recalled first.

For $N\geq 2$ define the following $N$-letter noncommutative alphabet $%
\subset \limfunc{End}\mathfrak{F}_{N}[q]$,%
\begin{equation}
u_{i}=\psi _{i+1}^{\ast }\psi _{i}~,\;\;i=1,\ldots ,N-1\quad \text{and}\quad
u_{N}=(-1)^{\limfunc{n}-1}q\psi _{1}^{\ast }\psi _{N}\;,
\end{equation}
then one has the following result \cite[Propostion 9.1]{ckcs}:
\begin{proposition}%[Korff, Stroppel \cite{ckcs}]
The (noncommutative) subalgebra in $\limfunc{End}\mathfrak{F}_{N}[q]$
generated by the $u_{i}$'s provides a faithful representation of the affine
nil-Temperley-Lieb algebra. That is, the following relations hold
\begin{equation}
u_{i}^{2}=u_{i}u_{i+1}u_{i}=u_{i+1}u_{i}u_{i+1}=0,\qquad
u_{i}u_{j}=u_{j}u_{i}\text{\quad if }|i-j|>1\func{mod}N\;,
\end{equation}%
where all indices are understood modulo $N$.
\end{proposition}

We now introduce special commuting elements in $%
\limfunc{End}\mathfrak{F}_{N}[q]$ which correspond to the elementary and
complete symmetric functions in the noncommutative alphabet $\mathcal{U}%
=\{u_{1},...,u_{N}\}$; compare with \cite{Postnikov} and \cite[Definition 9.4]{ckcs}

\begin{definition}[noncommutative symmetric polynomials]
For $r=1,2,...,N-1$ let%
\begin{equation}
\boldsymbol{e}_{r}=\sum_{|I|=r}\tprod_{i\in I}^{\circlearrowleft }u_{i}\text{%
\qquad and\qquad }\boldsymbol{h}_{r}=\sum_{|I|=r}\tprod_{i\in
I}^{\circlearrowright }u_{i},  \label{ncfunctions}
\end{equation}%
where $\tprod_{i\in I}^{\circlearrowright }u_{i}$ is the clockwise ordered
product of the letters $u_{i}$ such that if $i,i+1\in I$ the letter $u_{i+1}$
appears before $u_{i}$. The counterclockwise product $\tprod_{i\in
I}^{\circlearrowleft }u_{i}$ is obtained by reversing the previous cyclic
order. We also set $\boldsymbol{h}_{N}=(-1)^{\limfunc{n}-1}q$ and $%
\boldsymbol{e}_{N}|_{\mathfrak{F}_{n,N}[q]}=0$ except when $n=N$, where $%
\boldsymbol{e}_{N}|_{\mathfrak{F}_{N,N}[q]}=q$.
\end{definition}

In order to prove Theorem \ref{main} we will make use of the following
Proposition and Theorem which originally are due to Postnikov \cite%
{Postnikov}. An alternative proof of these facts using the particle picture
and the associated Clifford algebra can be found in \cite[Part II]{ckcs}.

\begin{proposition}
The elements in the set $\{\boldsymbol{e}_{r},\boldsymbol{h}_{s}\}$ pairwise
commute. Thus, the noncommutative Schur polynomials defined via the
(equivalent) determinant formulae
\begin{equation}
\boldsymbol{s}_{\lambda }=\det (\boldsymbol{e}_{\lambda
_{i}^{t}-i+j})_{1\leq i,j\leq N}=\det (\boldsymbol{h}_{\lambda
_{i}-i+j})_{1\leq i,j\leq N}\ .  \label{ncSchurdef}
\end{equation}%
satisfy all the familiar relations from the ring of commutative symmetric
functions. In particular, one has the specialisations $\boldsymbol{s}%
_{(1^{r})}=\boldsymbol{e}_{r}$ and $\boldsymbol{s}_{(r)}=\boldsymbol{h}_{r}$.
\end{proposition}

\begin{theorem}[Combinatorial quantum cohomology ring]
\label{Schurprod} Fix $n\in \mathbb{Z}_{\geq 0}$ and consider the $n$%
-particle subspace $\mathfrak{F}_{n,N}[q]\subset \mathfrak{F}_{N}[q]$. The
assignment
\begin{equation}
({\lambda },{\mu })\mapsto {\lambda }\star {\mu }:=\boldsymbol{s}_{\lambda }{%
\mu }  \label{freeSchurproduct}
\end{equation}%
for basis elements $\lambda {,\mu }\in \mathfrak{P}_{\leq n,k}$ turns $%
\mathfrak{F}_{n,N}[q]$ into a commutative, associative and unital $\mathbb{C}%
[q]$-algebra whose integral form is isomorphic to the quantum cohomology
ring $qH^{\ast }(\func{Gr}_{n,N})$. In particular, its structure constants
are given by the matrix elements of the noncommutative Schur polynomial, $%
\langle \nu ,\boldsymbol{s}_{\lambda }{\mu \rangle =}C_{\lambda ,\mu }^{\nu
,d}q^{d}$.
\end{theorem}

\begin{remark}{\rm
Comparing (\ref{freeSchurproduct}) with (\ref{fermi product}) one can convince oneself that the latter product formulation presents a simplification. For instance,
choosing $n=4,\;k=3$ and $\lambda =(2,2,1,0)$ according to (\ref%
{freeSchurproduct}) one first needs to compute the determinant in (\ref%
{ncSchurdef}), $s_{\lambda }=e_{2}e_{3}-e_{1}e_{4}$, and then multiply out
the elementary symmetric polynomials in the noncommutative alphabet $%
\{u_{1},\ldots ,u_{7}\}$ before acting with the individual monomial terms in
the expansion on the diagram $\mu $. Below we see the simplified computation
in terms of (\ref{fermi product}). However, the product description \eqref{freeSchurproduct} in terms
of noncommutative Schur polynomials is more convenient when proving
associativity; see \cite{ckcs}.
}
\end{remark}

As explained in \cite[Section 11]{ckcs} the main advantage of the fermion
formalism is that it allows to relate products in different quantum
cohomology rings and, thus, to successively create all rings $qH^{\ast }(%
\func{Gr}_{n,N})$ for $n=0,\ldots ,N$. The crucial result is the following
commutation relation of fermion creation and annihilation operators and
noncommutative Schur functions \cite[Proposition 11.4]{ckcs}.

\begin{proposition}%[Korff, Stroppel]
The following commutation relations hold true,%
\begin{equation}
\boldsymbol{s}_{\lambda }\psi _{i}^{\ast }=\sum_{r=0}^{\lambda _{1}}\psi
_{i+r}^{\ast }\sum_{\lambda /\mu =(r)}\boldsymbol{\bar{s}}_{\mu },
\label{Schurfermicomm}
\end{equation}%
where $\boldsymbol{\bar{s}}_{\mu }$ denotes the noncommutative Schur
polynomial (\ref{ncSchurdef}) with $q$ replaced by $-q$ and we impose again
the quasi-periodic boundary conditions $\psi _{j+N}^{\ast }=(-1)^{\limfunc{n}%
-1}q\psi _{j}^{\ast }$.
\end{proposition}
We now have collected all the necessary ingredients to prove the main result.
\begin{proof}[Proof of Theorem \ref{main} and derivation of \eqref{fermi product}]
Given any pair $\lambda {,\mu }\in \mathfrak{P}_{\leq n,k}$ denote by $%
w,w^{\prime }$ the associated 01-words in $\mathfrak{F}_{n,N}[q].$ Any word
can be written in the form $w=\psi _{\ell _{n}}^{\ast }\cdots \psi _{\ell
_{2}}^{\ast }\psi _{\ell _{1}}^{\ast }\varnothing $ with $1\leq \ell
_{n}<\ell _{n-1}<\cdots <\ell _{1}\leq N$. Thus, repeated application of the
above commutation relation \eqref{Schurfermicomm} yields%
\begin{multline*}
{\lambda }\star {\mu }=\boldsymbol{s}_{\lambda }{\mu =}\boldsymbol{s}%
_{\lambda }\psi _{\ell _{n}(\mu )}^{\ast }\cdots \psi _{\ell _{2}(\mu
)}^{\ast }\psi _{\ell _{1}(\mu )}^{\ast }\varnothing \\
=\sum_{\rho _{n-1}}\psi _{\ell _{n}(\mu )+|\lambda /\rho _{n-1}|}^{\ast }%
\boldsymbol{\bar{s}}_{\rho _{n-1}}\psi _{\ell _{n-1}(\mu )}^{\ast }\cdots
\psi _{\ell _{1}(\mu )}^{\ast }\varnothing \\
=\sum_{\rho _{n-2},\rho _{n-1}}\psi _{\ell _{n-1}(\mu )+|\lambda /\rho
_{n-1}|}^{\ast }\bar{\psi}_{\ell _{n-1}(\mu )+|\rho _{n-1}/\rho
_{n-2}|}^{\ast }\boldsymbol{s}_{\rho _{n-2}}\psi _{\ell _{n-2}(\mu )}^{\ast
}\cdots \psi _{\ell _{1}(\mu )}^{\ast }\varnothing \\
\vdots \\
=\sum_{(\rho _{n-1},\ldots ,\rho _{1})}\psi _{\ell _{n}(\mu )+|\lambda /\rho
_{n-1}|}^{\ast }\bar{\psi}_{\ell _{n-1}(\mu )+|\rho _{n-1}/\rho
_{n-2}|}^{\ast }\psi _{\ell _{n-2}(\mu )+|\rho _{n-2}/\rho _{n-3}|}^{\ast
}\cdots \varnothing ,
\end{multline*}%
where the sums run over all partitions $\rho _{i}$ such that $\rho
_{n}=\lambda $, $\rho _{0}=\emptyset $ and $\rho _{n+1-i}/\rho _{n-i}$ is a
horizontal strip. The constraint $\rho _{0}=\emptyset $ simply follows from
the fact that $\boldsymbol{s}_{\rho }\varnothing $ is only nonzero for $\rho
=\emptyset $. Such a sequence of partitions is equivalent to a
(semi-standard) tableau $T$, where $\rho _{i}$ is obtained by taking the
shape of $T$ after deleting all boxes with entries \TEXTsymbol{>} $i$.
Hence, the assertion (\ref{fermi product}) now follows from Theorem \ref%
{Schurprod}.
\end{proof}

\begin{remark}
{\rm In \cite[Proposition 11.4]{ckcs} a second commutation relation for the fermion
annihilation operators and the noncommutative Schur polynomials has been
derived. The latter leads to a product formula analogous to (\ref{fermi
product}) where one replaces the product of the fermion creation operators
acting on $w=00\cdots 0$ by a product in the $\psi _{i}$'s acting on the
word $w=11\cdots 1$. Since this formula can be easily obtained by applying
the parity and particle-hole duality transformations discussed in \cite[Section 8.3]{ckcs} we omit it here.}
\end{remark}

\begin{proof}[Proof of Corollary \ref{maincor}]
The proof is immediate as the product formula \eqref{fermi product} implies that the structure
constants are given by the following sum over matrix elements (`vacuum expectation values') in the Clifford algebra,%
\begin{equation}  \label{GWvev}
C_{\lambda \mu }^{\nu ,d}=\sum_{T=|\lambda |}(-1)^{d(n-1)}\langle
\varnothing ,\psi _{\ell _{1}(\nu )}\cdots \psi _{\ell _{n}(\nu )}\psi
_{\ell _{n}(\mu )+t_{n}}^{\ast }\cdots \psi _{\ell _{1}(\mu )+t_{1}}^{\ast
}\varnothing \rangle ,
\end{equation}%
where all indices are now understood modulo $N$ and $d=\#\{i~|~\ell _{i}(\mu
)+t_{i}>N\}$ because of \eqref{qpbc2}. For fixed weight vector $\alpha =(t_{1},\ldots ,t_{n})$ the
same matrix elements appears $K_{\lambda ,\alpha }$ times by definition of
the Kostka numbers. The resulting particle positions $\ell _{i}(\mu )+t_{i}$
must up to a permutation $\pi $ coincide with those of the partition $\nu $.
Since the fermion operators anticommute the sign factor $(-1)^{\ell (\pi )}$
follows and with it the asserted identity.
\end{proof}

\begin{example}
\label{fprodex}{\rm Set $N=7$ and $k=N-n=4$. Consider the partitions $%
\lambda =(2,2,1,0)$ and $\mu =(3,3,2,1)$. Converting $\mu $ into a $01$-word
we find for the positions of the $1$-letters $\ell (\mu )=(\ell _{1},\ldots
,\ell _{4})=(7,6,4,2)$. Writing down all Young tableaux of shape $\lambda $
with weight vectors $\alpha =(\alpha _{1},\ldots ,\alpha _{4})$ one finds
that the non-trivial contributions come from%
\begin{equation*}
\Yvcentermath1\underset{(9,7,5,3)}{{\young(11,24,3)}}{,\;}\underset{(9,7,5,3)%
}{{\young(11,23,4)}}{,\;}\underset{(9,8,4,3)}{{\young(11,22,4)}}{,\ }%
\underset{(9,6,5,4)}{{\young(11,34,4)}}{,\;}\underset{(8,7,6,3)}{{\young%
(13,24,3)}}{,\;}\underset{(8,7,6,3)}{{\young(12,33,4)}}{,\;}\underset{%
(7,8,6,3)}{{\young(22,33,4)}}{,\;}\underset{(8,7,5,4)}{{\young(13,24,4)}}{,\
}\underset{(8,7,5,4)}{{\young(12,34,4)}}{,\;}\underset{(7,8,5,4)}{{\young%
(22,34,4)}}{\;.}
\end{equation*}%
Below each tableau we have listed the resulting `particle positions' $\ell
_{i}^{\prime }=\ell _{i}(\mu )+t_{i}$ appearing in (\ref{fermi product}) as
indices of the fermion creation operators $\psi _{i}^{\ast }$. We have
dropped all those tableaux from the list for which two positions coincide.
For instance, the not listed Young tableau%
\begin{equation*}
\Yvcentermath1{\young(11,22,3)}
\end{equation*}%
yields after applying (\ref{qpbc2}) the following product of fermion
creation operators,%
\begin{equation*}
\psi _{\ell _{4}}^{\ast }\bar{\psi}_{\ell _{3}+1}^{\ast }\psi _{\ell
_{2}+2}^{\ast }\bar{\psi}_{\ell _{1}+2}^{\ast }=-q^{2}\psi _{2}^{\ast }\psi
_{5}^{\ast }\psi _{1}^{\ast }\psi _{2}^{\ast }=0\;.
\end{equation*}%
The latter vanishes because of the Clifford algebra relations (\ref{Cliff
relations}), which imply $(\psi _{i}^{\ast })^{2}=0$. In contrast the last
three tableaux listed above,
\begin{equation*}
\Yvcentermath1
\underset{(8,7,5,4)}{{\young(13,24,4)}}{,\
}\underset{(8,7,5,4)}{{\young(12,34,4)}}{,\;}\underset{(7,8,5,4)}{{\young%
(22,34,4)}}{\;,}
\end{equation*}%
 yield the same 01-word $w=1001101$ ($\lambda
(w)=(3,2,2,0)$) but with changing sign,%
\begin{multline*}
\psi _{\ell _{4}+2}^{\ast }\bar{\psi}_{\ell _{3}+1}^{\ast }\psi _{\ell
_{2}+1}^{\ast }\bar{\psi}_{\ell _{1}+1}^{\ast }\varnothing =\psi _{\ell
_{4}+2}^{\ast }\bar{\psi}_{\ell _{3}+1}^{\ast }\psi _{\ell _{2}+1}^{\ast }%
\bar{\psi}_{\ell _{1}+1}^{\ast }\varnothing = \\
-\psi _{\ell _{4}+2}^{\ast }\bar{\psi}_{\ell _{3}+1}^{\ast }\psi _{\ell
_{2}+2}^{\ast }\bar{\psi}_{\ell _{1}}^{\ast }\varnothing =q\psi _{1}^{\ast
}\psi _{4}^{\ast }\psi _{5}^{\ast }\psi _{7}^{\ast }\varnothing \;.
\end{multline*}%
The relevant Kostka numbers are $K_{\lambda ,(2,1,1,1)}=2\;$and $K_{\lambda
,(2,2,1,0)}=1$. Converting the other tableaux in the same manner into
01-words paying attention to the quasi-periodic boundary conditions (\ref%
{qpbc2}) we obtain the product expansion%
\begin{equation}\label{ppex}
\Yvcentermath1\yng(2,2,1)\star \yng(3,3,2,1)= \\
\Yvcentermath1q~\yng(2,2,2,1)+2q~\yng(3,2,1,1)+q~\yng(3,2,2)+q~\yng%
(3,3,1)+q^{2}~\emptyset \;.
\end{equation}%
}
\end{example}
\subsection{Comparison with the rim-hook algorithm}
Bertram, Ciocan-Fontanine and Fulton gave the following expression for
Gromov-Witten invariants in terms of Littlewood-Richardson coefficients \cite[p735]%
{BCF},%
\begin{equation}
C_{\lambda \mu }^{\nu ,d}=\sum_{\rho }\varepsilon (\rho /\nu )c_{\lambda \mu
}^{\rho },\qquad \varepsilon (\rho /\nu )=\prod_{i}(-1)^{(n-\text{width}%
(r_{i}))},  \label{GWrimhook}
\end{equation}%
where the sum runs over all Young diagrams $\rho $ which are obtained from $%
\nu $ by adding $d$ rim hooks $r_{i}$ each consisting of $N$ boxes and
starting in the first column. The integer $\text{width}(r_{i})$ is the number of columns the rim hook $r_i$ occupies; see \cite{BCF} for details.

Since Kostka numbers can be viewed as special cases of Littlewood-Richardson coefficients
\cite{Kingetal},%
\begin{equation}
K_{\lambda ,\alpha }=c_{\lambda \mu (\alpha )}^{\nu (\alpha )}\quad \text{%
with}\quad \mu _{i}(\alpha ):=\sum_{j>i}\alpha _{j},\quad \nu _{i}(\alpha
):=\sum_{j\geq i}\alpha _{j}~,  \label{KostkaLR}
\end{equation}%
one might ask whether the expression (\ref{GWKostka}) for Gromov-Witten
invariants coincides with the known expression from the rim-hook algorithm.

\begin{example}
\label{rimhookex1}{\rm To illustrate the algorithm we adopt Example 1
from \cite[p735]{BCF}. Set $k=n=5$ and $\lambda =(5,4,4,2,2)$, $\mu =(3,2,1)$%
, $\nu =(2,1)$. Then $d=(|\lambda |+|\mu |-|\nu |)/N=(17+6-3)/10=2$ and%
\begin{equation}
C_{\lambda \mu }^{\nu ,d}=c_{\lambda \mu }^{(5,5,4,3,2,2,2)}-c_{\lambda \mu
}^{(5,4,4,3,2,2,2,1)}=2-1=1\;.
\end{equation}%
In contrast let us determine the expression (\ref{GWKostka}) in terms of
Kostka numbers. For convenience we swap the roles of $\lambda $ and $\mu $
exploiting that the product is commutative. Converting $\lambda $ and $\nu $
into 01-words we find the following 1-letter positions $\ell (\lambda
)=(10,8,7,4,3)$ and $\ell (\nu )=(7,5,3,2,1)$. Since the weight vector $%
\alpha$ must obey the constraints $|\alpha|=|\mu |=6$ and $\alpha _{i}\leq
\mu _{1}=3$, there is only one possibility: $\alpha =(2,3,0,1,0)$ with Young
tableau%
\begin{equation*}
\Yvcentermath1{\young(112,22,4)\;.}
\end{equation*}%
Thus, we find the identity%
\begin{equation*}
C_{\lambda \mu }^{\nu ,d}=c_{\lambda \mu }^{(5,5,4,3,2,2,2)}-c_{\lambda \mu
}^{(5,4,4,3,2,2,2,1)}=K_{\mu ,(3,2,1)}=c_{(3,2,1),(3,1)}^{(6,3,1)}=1\;.
\end{equation*}
}
\end{example}

\begin{remark}{\rm
Equating expressions (\ref{GWKostka}) and (\ref{GWrimhook}) we
obtain in general non-trivial identities between Littlewood-Richardson
coefficients by employing the identity (\ref{KostkaLR}).}
\end{remark}

There exists a `dual rim hook algorithm' \cite{BCF} for which a simplified version has been stated; see  \cite{RRW}, \cite{Sottile} and \cite{Buch} for details. In the present context we wish to connect it with the free fermion formulation of the quantum cohomology ring and therefore briefly outline its derivation employing \eqref{STiso}.
\subsection{Comparison with the `dual rim hook algorithm'}\label{subsec:dualrimhook}
Given two partitions $\lambda ,\mu \in \mathfrak{P}_{\leq n,k}$ one exploits
(\ref{STiso}) to identify the Schubert classes $\sigma _{\lambda },\sigma
_{\mu }$ with the Schur polynomials $s_{\lambda },s_{\mu }$. In order to
compute the product in $qH^{\ast }(\func{Gr}_{n,N})$, one first performs the standard Littlewood-Richardson algorithm to obtain the (non-modified) expansion $s_{\lambda }s_{\mu }=\sum_{\nu
}c_{\lambda \mu }^{\nu }s_{\nu }$ where $c_{\lambda \mu }^{\nu }$ are again
the Littlewood-Richardson coefficients. Then one imposes the quotient
condition in (\ref{STiso}) by discarding all terms $s_{\nu }$ with
partitions $\nu $ which contain more than $n$ nonzero parts and by replacing
all the remaining Schur polynomials with the polynomials $%
(-1)^{d(n-1)}q^{d}s_{v(\nu )},\;Nd=|\lambda |-|v|$ where $v(\nu )$ is the
unique set of integers such that%
\begin{equation}
v_{i}(\nu )=\nu _{i}\func{mod}N\qquad \text{and\qquad }i-n\leq v_{i}(\nu
)<i+k\;.  \label{GWreduce}
\end{equation}%
Recall that Schur polynomials can be defined for arbitrary vectors $v\in
\mathbb{Z}^{n}$ exploiting the relations \cite{MacDonald}
\begin{equation}\label{Schurstraight}
s_{(\ldots ,a,b,\ldots
)}=-s_{(\ldots ,b-1,a+1,\ldots )}\quad\text{and}\quad s_{(\ldots ,a,a+1,\ldots )}=0.
\end{equation}
Collecting terms the resulting coefficients in the expansion are the Gromov-Witten invariants.
\begin{proof}[Derivation of the algorithm in the free fermion picture]
As the $q$-dependence is trivially deduced from the degree of the Gromov-Witten invariant we set $q=1$ for simplicity. Let $\mathcal{I}=\langle h_{k+1},\ldots,h_{n+k-1},h_{n+k}+(-1)^n\rangle$ be the two-sided ideal specified in \eqref{STiso} and consider the extension of $qH^{\ast }(\func{Gr}_{n,n+k})$ to the complex numbers $\mathbb{C}$. Similar as in \cite[Proof of Theorem 6.20]{ckcs} one shows that $\mathcal{I}$ is radical. Recall from \cite[Section 10]{ckcs} that the affine variety $\mathbb{V}(\mathcal{I})\subset\mathbb{C}^n$ determined by this ideal is as a set equal to the solutions of the following system of equations, $y_{i}^{n+k}=(-1)^{n-1}$, $i=1,\ldots n$, which are the free fermion Bethe Ansatz equations and can be solved
explicitly. Employing Hilbert's Nullstellensatz, namely that $\mathcal{I}$ equals the ideal of polynomials which vanish on $\mathbb{V}(\mathcal{I})$, one then shows that two functions $f,g\in\mathbb{C}[e_1,\ldots,e_n]/\mathcal{I}$ coincide if and only if they have equal values on the set of solutions of the Bethe Ansatz equations. Thus, since there are only $n$ variables, $s_{\nu }(y_{1},\ldots,y_{n})=0$ for partitions of length $\ell (\nu )>n$. Given a Schur polynomial $s_{\nu }$ with $\nu _{1}>k$ in the Littlewood-Richardson expansion one may write%
\begin{eqnarray}
s_{\nu }(y_{1},\ldots ,y_{n}) &=&\sum_{\pi \in S_{n}}\pi \left( y_{1}^{\nu
_{1}}\cdots y_{n}^{\nu _{n}}\tprod_{i<j}\frac{1}{1-y_{j}/y_{i}}\right) \label{Schurpolyrep}\\
&=&(-1)^{n-1}s_{(\nu _{1}-n-k,\nu_{2},\ldots ,\nu _{n})}(y_{1},\ldots
,y_{n})\;,\notag
\end{eqnarray}%
where the first equality is a standard identity which can be found in e.g. \cite{MacDonald}. The second identity has used the Bethe Ansatz equations. Repeating this procedure we can achieve the above identity $s_{\nu }(y)=
(-1)^{d(n-1)}s_{v(\nu )}(y)$ with $v(\nu)$ given by \eqref{GWreduce}.
\end{proof}

\begin{example}
\label{rimhookex2}{\rm Set $n=3,\,k=4$ and consider the partitions $%
\lambda =(3,1,0)$ and $\mu =(3,2,0).$ The Littlewood-Richardson rule yields the
partitions%
\begin{multline}
\nu =(6,3,0),(6,2,1),(5,4,0),(5,3,1),(5,3,1),(5,2,2),(4,4,1),  \label{LRRex} \\
(4,3,2),(4,3,2),(3,3,3),(5,2,1,1),(4,3,1,1),(4,2,2,1),(3,3,2,1)\;,
\end{multline}%
from which we can remove the last four as they have length \TEXTsymbol{>}
3. From the remaining terms we only need to transform those with partitions outside the bounding box, i.e. those for which $\nu_1>k$. For instance, consider $\nu =(5,3,1)$ then $v(\nu )=(-2,3,1)$ and $s_{\nu
}=qs_{(-2,3,1)}=-qs_{(2,-1,1)}=qs_{(2,0,0)}$. Similarly, we find for $\nu
=(6,3,0)$ that $s_{\nu }=-qs_{(2,0,0)}$, whence%
\begin{equation*}
C_{\lambda \mu }^{(2,0,0),1}=c_{\lambda \mu }^{(5,3,1)}-c_{\lambda \mu
}^{(6,3,0)}=2-1=1\;.
\end{equation*}%
In comparison, the expression in terms of Kostka numbers (\ref{GWKostka})
with $\ell (\mu )=(6,4,1)$ and $\ell (\nu )=(5,2,1)$ is%
\begin{equation*}
C_{\lambda \mu }^{(2,0,0),1}=K_{\lambda (2,1,1)}-K_{\lambda (3,1,0)}=2-1=1\;.
\end{equation*}%
The full product expansion reads%
\begin{equation*}
\Yvcentermath1\yng(3,1)\star \yng(3,2)= \\
\Yvcentermath1q~\yng(2)+q~\yng(1,1)+~\yng(4,4,1)+2~\yng(4,3,2)+~\yng%
(3,3,3)\;.
\end{equation*}%
}
\end{example}

\section{Proof and examples of the projection formula}\label{sec:projection}

As a preliminary step to the proof we again need to recall some technical results
first. Throughout this section set $\zeta =\exp \frac{2\pi i}{k+n}$ and
(following \cite{Rietsch}) define a map $I:\mathfrak{P}_{\leq
n,k}\rightarrow (\frac{1}{2}\mathbb{Z)}^{n}$ through%
\begin{equation}\label{Imap}
\sigma \mapsto I(\sigma )=\left( \tfrac{n+1}{2}+\sigma _{n}-n,\ldots ,\tfrac{%
n+1}{2}+\sigma _{1}-1\right) \;.
\end{equation}%
By abuse of notation we use the same symbol to denote the analogous map
where $k$ and $n$ are interchanged, i.e. for $\sigma ^{t}$, the transpose of
$\sigma $, we set
\begin{equation}\label{Itmap}
I(\sigma ^{t})=\left( \tfrac{k+1}{2}+\sigma _{k}^{t}-k,\ldots ,\tfrac{k+1}{2}%
+\sigma _{1}^{t}-1\right) .
\end{equation}

\begin{lemma}[Rietsch]
\label{technical}Let $s_{\lambda }$ denote the Schur polynomial associated
with the partition $\lambda $, then one has the identity%
\begin{equation}
s_{\lambda }(\zeta ^{I(\sigma )})=s_{\lambda ^{t}}(\zeta ^{-I(\sigma
^{t})})\;.
\end{equation}
\end{lemma}

Recall the bijection $P_{k}^{+}\rightarrow \mathfrak{P}_{\leq n-1,k}$
defined in (\ref{weight2part}) for $\widehat{\mathfrak{sl}}(n)_{k}$ weights,
i.e. each $\hat{\lambda}\in P_{k}^{+}$ uniquely corresponds to a partition $%
\lambda \in \mathfrak{P}_{\leq n-1,k}$ whose associated Young diagram has at
most $n-1$ rows and $k$ columns. Likewise there exists a bijection between
the $\widehat{\mathfrak{sl}}(k)_{n}$ weights, which we denote by $\tilde{P}%
_{+}^{n}$, and $\mathfrak{P}_{\leq k-1,n}$. Employing these two bijections
we shall henceforth label the $\widehat{\mathfrak{sl}}(n)_{k}$-fusion
coefficients $\mathcal{N}$ and the $\widehat{\mathfrak{sl}}(k)_{n}$-fusion
coefficients $\mathcal{\tilde{N}}$ by the respective partitions rather than
the affine weights to unburden the notation, e.g.%
\begin{equation}
\mathcal{N}_{\hat{\lambda}\hat{\mu}}^{(k)\hat{\nu}}\rightarrow \mathcal{N}%
_{\lambda \mu }^{\nu },\qquad \hat{\lambda},\hat{\mu},\hat{\nu}\in
P_{k}^{+}\;\;\text{and\ \ }\lambda ,\mu ,\nu \in \mathfrak{P}_{\leq n-1,k}\;.
\end{equation}

Given a weight $\hat{\lambda}=\sum_{i}m_{i}\hat{\omega}_{i}$ in $P_{k}^{+}$
the map $\limfunc{rot}:P_{k}^{+}\rightarrow P_{k}^{+}$ defined via%
\begin{equation}
\hat{\lambda}\mapsto \limfunc{rot}\hat{\lambda}:=\sum_{i\in \mathbb{Z}%
_{n}}m_{i+1}\hat{\omega}_{i}  \label{rotdef}
\end{equation}%
is the Dynkin diagram automorphism of $\widehat{\mathfrak{sl}}(n)$ of order $%
n$. We denote by $\widetilde{\limfunc{rot}}$ its $\widehat{\mathfrak{sl}}(k)$
counterpart. Because of the bijection between weights and partitions
respectively Young diagrams the automorphism (\ref{rotdef}) induces a map $%
\limfunc{rot}:\mathfrak{P}_{\leq n-1,k}\rightarrow \mathfrak{P}_{\leq n-1,k}$
which acts by adding a top row of width $k$ and then deleting all columns of
height $n$ in the resulting diagram. In the course of our discussion we will
also need the $\mathfrak{\widehat{gl}}(n+k)$ automorphism $\func{Rot}$ which
we interpret as the following map $\mathfrak{P}_{\leq n,k}\rightarrow
\mathfrak{P}_{\leq n,k}$,
\begin{equation}
\limfunc{Rot}(\lambda )=\left\{
\begin{array}{cc}
(\lambda _{1}-1,\ldots ,\lambda _{k}-1), & \text{if }w_{1}(\lambda )=0 \\
(n,\lambda _{1},\lambda _{2},\ldots ,\lambda _{k-1}), & \text{if }%
w_{1}(\lambda )=1%
\end{array}%
\right. \;.  \label{Rot}
\end{equation}

To prove the identity (\ref{curious}) we also employ the known expressions
for the Gromov-Witten invariants and fusion coefficients in terms of the
Bertram-Vafa-Intrilligator (see \cite{Rietsch} for this particular explicit
presentation),%
\begin{equation}
C_{\lambda \mu }^{\nu ,d}=\frac{1}{(k+n)^{n}}\sum_{\sigma \in \mathfrak{P}%
_{\leq n,k}}\frac{s_{\lambda }(\zeta ^{-I(\sigma )})s_{\mu }(\zeta
^{-I(\sigma )})s_{\nu }(\zeta ^{I(\sigma )})}{\prod_{i<j}|\zeta
^{I_{i}(\sigma )}-\zeta ^{I_{j}(\sigma )}|^{-2}}~,  \label{BVI}
\end{equation}
and the Verlinde formula for the $\mathfrak{\widehat{sl}}(n)_k$-WZNW fusion
ring (see e.g. \cite{CFTbook}),%
\begin{equation}
\mathcal{N}_{\hat{\lambda}\hat{\mu}}^{(k)\hat{\nu}}=\frac{1}{n(k+n)^{n-1}}%
\sum_{\sigma \in \mathfrak{P}_{\leq n-1,k}}\frac{s_{\lambda }(\zeta
^{|\sigma |}\zeta ^{-I(\sigma )})s_{\mu }(\zeta ^{|\sigma |}\zeta
^{-I(\sigma )})s_{\nu }(\zeta ^{-|\sigma |}\zeta ^{I(\sigma )})}{%
\prod_{i<j}^{n}|\zeta ^{I_{i}(\sigma )}-\zeta ^{I_{j}(\sigma )}|^{-2}}
\label{Verlinde}
\end{equation}
respectively. Here we have used in (\ref{Verlinde}) the Kac-Peterson formula
for the modular S-matrix \cite{KacPeterson}. Both expressions can be derived
in a purely combinatorial setting; see \cite{ckcs}.

We split the proof of the formula (\ref{curious}) into two parts. First we
show the following:

\begin{claim}
Let $\lambda ,\mu ,\nu \in \mathfrak{P}_{\leq n,k}$ then the Gromov-Witten
invariant $C_{\lambda \mu }^{\nu ,d}$ can be written as the following sum of
two fusion coefficients%
\begin{equation}
C_{\lambda \mu }^{\nu ,d}=\frac{n}{k+n}~\mathcal{N}_{\lambda ^{\prime }\mu
^{\prime }}^{\limfunc{rot}^{d}(\nu ^{\prime })}+\frac{k}{k+n}\mathcal{\tilde{%
N}}_{(\lambda ^{\prime \prime })^{t}(\mu ^{\prime \prime })^{t}}^{\widetilde{%
\limfunc{rot}}^{d}(\nu ^{\prime \prime })^{t}}\;,  \label{GWdecomp}
\end{equation}%
where $\limfunc{rot},\widetilde{\limfunc{rot}}$ are the Dynkin diagram
automorphisms introduced earlier.
\end{claim}

\begin{proof}
To derive (\ref%
{GWdecomp}) we start by writing the Bertram-Vafa-Intrilligator formula (\ref%
{BVI}) for $C_{\lambda \mu }^{\nu ,d}$ as two separate sums,%
\begin{multline}
C_{\lambda ,\mu }^{\nu ,d}=\frac{1}{(k+n)^{n}}\sum_{\sigma \in \mathfrak{P}%
_{\leq n-1,k}}\frac{s_{\lambda }(\zeta ^{-I(\sigma )})s_{\mu }(\zeta
^{-I(\sigma )})s_{\nu }(\zeta ^{I(\sigma )})}{\prod_{i<j}|\zeta
^{I_{i}(\sigma )}-\zeta ^{I_{j}(\sigma )}|^{-2}} \\
+\frac{1}{(k+n)^{n}}\sum_{\sigma \in \mathfrak{P}_{\leq n,k}/\mathfrak{P}%
_{\leq n-1,k}}\frac{s_{\lambda ^{t}}(\zeta ^{I(\sigma ^{t})})s_{\mu
^{t}}(\zeta ^{I(\sigma ^{t})})s_{\nu ^{t}}(\zeta ^{-I(\sigma ^{t})})}{%
\prod_{i<j}|\zeta ^{I_{i}(\sigma )}-\zeta ^{I_{j}(\sigma )}|^{-2}}\; ,
\label{BVI2}
\end{multline}%
i.e. the first sum runs over all partitions $\sigma $ in the bounding box of
height $n-1$ and width $k$ and the second sum over all $\sigma $ which have $%
n$ rows and at most $k$ columns.

We are now rewriting the first sum with $\sigma \in \mathfrak{P}_{\leq
n-1,k} $ as a fusion coefficient. Observe that the Pieri-formula for Schur
polynomials \cite{MacDonald} implies that multiplying $s_{\nu ^{\prime
}}(\zeta ^{I(\sigma )})$ with $e_{n}(\zeta ^{I(\sigma )})=\zeta ^{|\sigma |}$
yields the Schur polynomial with the partition $\nu ^{\prime }$ plus an
additional $n$-column. Hence, $s_{\nu }(\zeta ^{I(\sigma )})=\zeta ^{|\sigma
|m_{n}(\nu )}s_{\nu ^{\prime }}(\zeta ^{I(\sigma )})$, where $m_{n}(\nu )$
is the number of $n$-columns in $\nu $. Similarly, it follows from the second
Pieri formula for Schur polynomials that multiplying $s_{\nu ^{\prime }}(\zeta ^{I(\sigma )})$
with
\begin{equation*}
h_{k}(\zeta ^{I(\sigma )})=s_{(k)}(\zeta ^{I(\sigma )})=s_{(1^{k})}(\zeta
^{-I(\sigma ^{t})})=e_{k}(\zeta ^{-I(\sigma ^{t})})=\zeta ^{-|\sigma |}
\end{equation*}%
corresponds to adding a row of width $k$ to $\nu ^{\prime }$. Here we have
used Lemma \ref{technical} to rewrite the $k^{\text{th}}$ complete symmetric polynomial. Both formulae then allow us to describe the
action of the Dynkin diagram automorphism $\limfunc{rot}$ (adding a $k$-row
and then removing all columns with $n$-boxes),%
\begin{eqnarray*}
s_{\limfunc{rot}^{d}(\nu ^{\prime })}(\zeta ^{I(\sigma )}) &=&\zeta
^{-|\sigma |d-(m_{n-1}(\nu )+\cdots +m_{n-d}(\nu ))|\sigma |}s_{\nu ^{\prime
}}(\zeta ^{I(\sigma )}) \\
&=&\zeta ^{-\frac{k+n}{n}|\sigma |d}\zeta ^{-\frac{|\nu ^{\prime }|-|%
\limfunc{rot}^{d}(\nu ^{\prime })|}{n}|\sigma |}s_{\nu ^{\prime }}(\zeta
^{I(\sigma )}),
\end{eqnarray*}%
where
\begin{eqnarray*}
|\limfunc{rot}\nolimits^{d}(\nu ^{\prime })| &=&|\nu |+kd-n(m_{n}(\nu
)+\cdots +m_{n-d}(\nu )) \\
&=&|\nu ^{\prime }|+kd-n(m_{n-1}(\nu )+\cdots +m_{n-d}(\nu ))\;.
\end{eqnarray*}%
Therefore, we may rewrite the product of Schur functions in the first sum of
(\ref{BVI2}) as%
\begin{multline*}
s_{\lambda }(\zeta ^{-I(\sigma )})s_{\mu }(\zeta ^{-I(\sigma )})s_{\nu
}(\zeta ^{I(\sigma )})= \\
\zeta ^{\frac{k+n}{n}|\sigma |d}\zeta ^{\frac{|\nu |-|\limfunc{rot}%
\nolimits^{d}(\nu ^{\prime })|}{n}|\sigma |}s_{\lambda }(\zeta ^{-I(\sigma
)})s_{\mu }(\zeta ^{-I(\sigma )})s_{\limfunc{rot}\nolimits^{d}(\nu ^{\prime
})}(\zeta ^{I(\sigma )})= \\
\zeta ^{\frac{|\lambda |+|\mu |-|\limfunc{rot}\nolimits^{d}(\nu ^{\prime })|%
}{n}|\sigma |}s_{\lambda }(\zeta ^{-I(\sigma )})s_{\mu }(\zeta ^{-I(\sigma
)})s_{\limfunc{rot}\nolimits^{d}(\nu ^{\prime })}(\zeta ^{I(\sigma )})= \\
\zeta ^{\frac{|\lambda ^{\prime }|+|\mu ^{\prime }|-|\limfunc{rot}%
\nolimits^{d}(\nu ^{\prime })|}{n}|\sigma |}s_{\lambda ^{\prime }}(\zeta
^{-I(\sigma )})s_{\mu ^{\prime }}(\zeta ^{-I(\sigma )})s_{\limfunc{rot}%
\nolimits^{d}(\nu ^{\prime })}(\zeta ^{I(\sigma )})\;.
\end{multline*}%
Thus, the first sum in the decomposition (\ref{BVI2}) becomes%
\begin{multline*}
\frac{1}{(k+n)^{n}}\sum_{\sigma \in \mathfrak{P}_{\leq n-1,k}}\frac{%
s_{\lambda }(\zeta ^{-I(\sigma )})s_{\mu }(\zeta ^{-I(\sigma )})s_{\nu
}(\zeta ^{I(\sigma )})}{\prod_{i<j}^{n}|\zeta ^{I_{i}(\sigma )}-\zeta
^{I_{j}(\sigma )}|^{-2}}= \\
\frac{1}{(k+n)^{n}}\sum_{\sigma \in \mathfrak{P}_{\leq n-1,k}}\zeta
^{|\sigma |\frac{|\lambda ^{\prime }|+|\mu ^{\prime }|-|\limfunc{rot}%
\nolimits^{d}(\nu ^{\prime })|}{n}}\frac{s_{\lambda ^{\prime }}(\zeta
^{-I(\sigma )})s_{\mu ^{\prime }}(\zeta ^{-I(\sigma )})s_{\limfunc{rot}%
\nolimits^{d}(\nu ^{\prime })}(\zeta ^{I(\sigma )})}{\prod_{i<j}^{n}|\zeta
^{I_{i}(\sigma )}-\zeta ^{I_{j}(\sigma )}|^{-2}}= \\
\frac{n}{k+n}~\mathcal{N}_{\lambda ^{\prime }\mu ^{\prime }}^{(k),\limfunc{%
rot}\nolimits^{d}(\nu ^{\prime })}.
\end{multline*}%
Let us now turn to the second sum in (\ref{BVI2}) where the Young diagram of
$\sigma \in \mathfrak{P}_{\leq n,k}/\mathfrak{P}_{\leq n-1,k}$ has height $n$%
. Set $\tilde{\sigma}^{t}=(\sigma _{1}-1,\ldots ,\sigma _{n}-1)\in \mathfrak{%
P}_{\leq n,k-1}$ then we calculate using Lemma \ref{technical}%
\begin{multline*}
\frac{1}{(k+n)^{n}}\sum_{\sigma \in \mathfrak{P}_{\leq n,k}/\mathfrak{P}%
_{\leq n-1,k}}\frac{s_{\lambda }(\zeta ^{-I(\sigma )})s_{\mu }(\zeta
^{-I(\sigma )})s_{\nu }(\zeta ^{I(\sigma )})}{\prod_{i<j}^{n}|\zeta
^{I_{i}(\sigma )}-\zeta ^{I_{j}(\sigma )}|^{-2}}= \\
\frac{k}{k+n}\frac{1}{k(k+n)^{k-1}}\sum_{\tilde{\sigma}\in \mathfrak{P}%
_{\leq k-1,n}}\frac{s_{\lambda ^{t}}(\zeta ^{I(\tilde{\sigma})})s_{\mu
^{t}}(\zeta ^{I(\tilde{\sigma})})s_{\nu ^{t}}(\zeta ^{-I(\tilde{\sigma})})}{%
\prod_{i<j}^{k}|\zeta ^{I_{i}(\tilde{\sigma})}-\zeta ^{I_{j}(\tilde{\sigma}%
)}|^{-2}},
\end{multline*}%
where we have used that%
\begin{equation*}
\frac{1}{k(k+n)^{n-1}}\prod_{1\leq i<j\leq n-1}|\zeta ^{I_{i}(\sigma
)}-\zeta ^{I_{j}(\sigma )}|^{2}=\frac{1}{k(k+n)^{k-1}}\prod_{1\leq i<j\leq
k}|\zeta ^{I_{i}(\sigma ^{t})}-\zeta ^{I_{j}(\sigma ^{t})}|^{2}
\end{equation*}%
as well as%
\begin{equation*}
s_{\lambda }(\zeta ^{-I(\sigma )})s_{\mu }(\zeta ^{-I(\sigma )})s_{\nu
}(\zeta ^{I(\sigma )})=s_{\lambda ^{t}}(\zeta ^{I(\tilde{\sigma})})s_{\mu
^{t}}(\zeta ^{I(\tilde{\sigma})})s_{\nu ^{t}}(\zeta ^{-I(\tilde{\sigma})}).
\end{equation*}%
The second identity is true because of Lemma \ref{technical} and because $%
|\lambda |+|\mu |-|\nu |=0\func{mod}N$. Recall that the Gromov-Witten
invariant is only nonzero provided this last condition holds. Hence, we see
that the second sum in \eqref{GWdecomp} is of the same form as the first sum
upon interchanging $n$ and $k$ and taking the complex conjugate. Thus,
running through exactly the same computation as before with $k$ and $n$
interchanged we arrive at (\ref{GWdecomp}) by exploiting the reality of the
fusion coefficients, $\mathcal{\tilde{N}}_{\lambda \mu }^{\nu }=\overline{%
\mathcal{\tilde{N}}_{\lambda \mu }^{\nu }}$. The latter follows here from
the definition (\ref{fusiondef}) of the fusion ring where we identified the
structure constants with dimensions of certain moduli spaces, but it can
also be proved combinatorially; see \cite{ckcs}.
\end{proof}

The identity (\ref{GWdecomp}) between Gromov-Witten invariants and fusion
coefficients can be further simplified to obtain (\ref{curious}) by
exploiting rotation invariance and the known level-rank duality between the
fusion coefficients of the $\widehat{\mathfrak{sl}}(n)_{k}$ and $\widehat{%
\mathfrak{sl}}(k)_{n}$-Verlinde algebras.

\begin{proposition}[rotation invariance \& level-rank duality]
\label{fusionsymm}One has the following equalities:

\begin{itemize}
\item[(i)] $\mathcal{N}_{\limfunc{rot}(\lambda )\mu }^{\nu }=\mathcal{N}%
_{\lambda \limfunc{rot}(\mu )}^{\nu }=\mathcal{N}_{\lambda \mu }^{\limfunc{%
rot}^{-1}(\nu )}$

\item[(ii)] Let $\lambda ,\mu ,\nu \in \mathfrak{P}_{\leq n-1,k}$ and $%
(\lambda ^{t})^{\prime },(\mu ^{t})^{\prime },(\nu ^{t})^{\prime }\in
\mathfrak{P}_{\leq k-1,n}$ be the $k$-column reductions of the transposed
partitions, then%
\begin{equation*}
\mathcal{N}_{\lambda \mu }^{\nu }=\mathcal{\tilde{N}}_{(\lambda
^{t})^{\prime }(\mu ^{t})^{\prime }}^{\widetilde{\limfunc{rot}}^{\hat{d}%
}(\nu ^{t})^{\prime }},\qquad \hat{d}=\frac{|\lambda |+|\mu |-|\nu |}{n}~.
\end{equation*}
\end{itemize}
\end{proposition}

We shall accept these results here without proof. Their derivations can be
found in textbooks, see e.g. \cite{CFTbook}.

\begin{proof}[Proof of Proposition \protect\ref{mainprop}]
Given the previous result (\ref{GWdecomp}), the identity (\ref{curious})
follows from Prop \ref{fusionsymm} (ii) provided we can show that
\begin{equation*}
\mathcal{N}_{\lambda ^{\prime }\mu ^{\prime }}^{\limfunc{rot}^{d}(\nu
^{\prime })}=\mathcal{\tilde{N}}_{(\lambda ^{\prime \prime })^{t}(\mu
^{\prime \prime })^{t}}^{\widetilde{\limfunc{rot}}^{d}(\nu ^{\prime \prime
})^{t}}\;.
\end{equation*}%
Let $m_{n}(\nu )$ denote the multiplicity of $n$-columns in the Young
diagram of $\nu $. Then we have the identity%
\begin{equation}
\mathfrak{P}_{\leq k-1,n-1}\ni ((\nu ^{\prime })^{t})^{\prime }=\widetilde{%
\limfunc{rot}}^{-m_{n}(\nu )}(\nu ^{\prime \prime })^{t}\;,
\label{doublereduction}
\end{equation}%
where the first prime indicates removal of all $n$-columns and the second
one removal of all $k$-columns. If $m_{n}(\nu )=0$ this identity is obvious.
For $m_{n}(\nu )>0$ the removal of a column of $n$-boxes in the Young
diagram of $\nu $ obviously corresponds to subtracting a row of $n$-boxes in
the transpose diagram. The latter corresponds to the action of $\widetilde{%
\limfunc{rot}}^{-1}$. This generalizes to the relationship%
\begin{equation}
\sigma :=((\limfunc{rot}\nolimits^{d}(\nu ^{\prime })^{t})^{^{\prime }}=%
\widetilde{\limfunc{rot}}^{-m_{n}(\nu )-\cdots -m_{n-d}(\nu )}(\nu ^{\prime
\prime })^{t},  \label{rotten}
\end{equation}%
because each action of the automorphism $\limfunc{rot}$ involves a column
reduction after adding a row of $k$-boxes. Observing that%
\begin{eqnarray*}
m_{n}(\nu )+\cdots +m_{n-d}(\nu ) &=&\frac{k}{n}~d+\frac{|\nu |-|\limfunc{rot%
}\nolimits^{d}(\nu ^{\prime })|}{n} \\
&=&\hat{d}^{\prime }+m_{n}(\lambda )+m_{n}(\mu )-d
\end{eqnarray*}%
with%
\begin{equation*}
\hat{d}^{\prime }=\frac{|\lambda ^{\prime }|+|\mu ^{\prime }|-|\limfunc{rot}%
\nolimits^{d}(\nu ^{\prime })|}{n}\quad \;\text{and}\quad \text{\ }d=\frac{%
|\lambda |+|\mu |-|\nu |}{k+n}
\end{equation*}%
we find by employing (\ref{doublereduction}), Proposition \ref{fusionsymm}
and then \eqref{rotten}
\begin{equation}
\mathcal{N}_{\lambda ^{\prime }\mu ^{\prime }}^{\func{rot}^{d}(\nu ^{\prime
})}\overset{\eqref{doublereduction}}{=}\mathcal{\tilde{N}}_{\widetilde{%
\limfunc{rot}}^{-m_{n}(\lambda )}(\lambda ^{\prime \prime })^{t}\widetilde{%
\limfunc{rot}}^{-m_{n}(\mu )}(\mu ^{\prime \prime })^{t}}^{\widetilde{%
\limfunc{rot}}^{\hat{d}^{\prime }}(n)}\overset{\text{Prop}\;\ref{fusionsymm}%
,\,\eqref{rotten}}{=}\mathcal{\tilde{N}}_{(\lambda ^{\prime \prime
})^{t}(\mu ^{\prime \prime })^{t}}^{\widetilde{\limfunc{rot}}^{d}(\nu
^{\prime \prime })^{t}}
\end{equation}%
the desired relation.
\end{proof}

\begin{example}\label{projectionex}{\rm
Consider the product expansion \eqref{ppex} in the quantum cohomology ring from the previous Example \ref{fprodex}. Applying the reduction
formula (\ref{curious}) for $\widehat{\mathfrak{sl}}(n)_{k}$ we find that $%
\lambda ^{\prime }=(2,2,1,0)=\mu =\mu ^{\prime }$ and%
\begin{multline*}
\Yvcentermath1\yng(2,2,1)~\ast ~\yng(2,2,1)= \\
\Yvcentermath1\func{rot}\yng(1,1,1)+2\func{rot}\yng(2,1)+\func{rot}\yng%
(3,2,2)+\func{rot}\yng(3,3,1)+\func{rot}^{2}(\emptyset )= \\
\Yvcentermath1\yng(2)+2~\yng(3,2,1)+\yng(1,1)+\yng(2,2,2)+\yng(3,3)\;.
\end{multline*}%
On the other hand if we apply instead the reduction to the $\widehat{%
\mathfrak{sl}}(k)_{n}$ fusion product $\tilde{\ast}$ with $\lambda
\rightarrow (\lambda ^{\prime \prime })^{t}=(2,1)$ and $\mu \rightarrow (\mu
^{\prime \prime })^{t}=(3,2)$ the product expansion becomes%
\begin{equation*}
\Yvcentermath1\yng(2,1)~\tilde{\ast}~\yng(3,2)= \\
\Yvcentermath1\yng(1,1)+2~\yng(3,2)+~\yng(2)+~\yng(4,1)+~\yng(4,4)\;.
\end{equation*}%
}
\end{example}

%\begin{remark}{\rm
Implicit in the projection formula (\ref{curious}) of Gromov-Witten
invariants onto fusion coefficients is the statement that there exist
identities between different $C_{\lambda \mu }^{\nu ,d}$. Namely, for $%
\lambda ,\mu \in \mathfrak{P}_{\leq n,k}/\mathfrak{P}_{\leq n-1,k}$ the
product expansions of $\sigma _{\lambda }\star \sigma _{\mu },~\sigma
_{\lambda ^{\prime }}\star \sigma _{\mu },~\sigma _{\lambda }\star \sigma
_{\mu ^{\prime }}$ and $\sigma _{\lambda ^{\prime }}\star \sigma _{\mu
^{\prime }}$ in $qH^{\ast }(\limfunc{Gr}_{n,n+k})$ all get mapped onto the
same fusion product expansion in $V_{k}(\widehat{\mathfrak{sl}}(n);\mathbb{Z}%
)$, because according to (\ref{curious}) there must exist $\nu _{1},\ldots
,\nu _{4}$ and $d_{1},\ldots ,d_{4}$ such that%
\begin{equation}
C_{\lambda \mu }^{\nu _{1},d_{1}}=C_{\lambda ^{\prime }\mu }^{\nu
_{2},d_{2}}=C_{\lambda \mu ^{\prime }}^{\nu _{3},d_{3}}=C_{\lambda ^{\prime
}\mu ^{\prime }}^{\nu _{4},d_{4}}=\mathcal{N}_{\lambda ^{\prime }\mu
^{\prime }}^{\nu },
\end{equation}%
where $\nu =\func{rot}^{d_{i}}(\nu _{i}^{\prime })$ for $i=1,\ldots ,4$. In
fact, using rotation invariance of the Gromov-Witten invariants \cite[Prop 11.1 (3)]{ckcs},
$C_{\limfunc{Rot}(\lambda ),\mu }^{\nu ,d}=C_{\lambda ,\limfunc{Rot}(\mu
)}^{\nu ,d^{\prime }}=C_{\lambda ,\mu }^{\limfunc{Rot}\nolimits^{-1}(\nu
),d^{\prime \prime }}$, it is easy to show that%
\begin{equation}
\nu _{2}=\limfunc{Rot}\nolimits^{m_{n}(\lambda )}(\nu _{1}),\;\nu _{3}=%
\limfunc{Rot}\nolimits^{m_{n}(\mu )}(\nu _{1}),\;\nu _{4}=\limfunc{Rot}%
\nolimits^{m_{n}(\lambda )+m_{n}(\mu )}(\nu _{1}),
\end{equation}%
where $m_{n}(\lambda ),m_{n}(\mu )$\ are the number of $n$-columns in $%
\lambda ,\mu $ and the respective degrees $d_{i}$ can be computed according
to the formula $|\limfunc{Rot}\nolimits^{a}(\lambda )|=|\lambda |+N\limfunc{n%
}_{a}(\lambda )-an$ and $Nd_{1}=|\lambda |+|\mu |-|\nu _{1}|$.\\
%}\end{remark}

The reduction formula (\ref{curious}) can be inverted.

\begin{corollary}[`Lifting' of fusion coefficients]
\label{lift}Let $\hat{\lambda},\hat{\mu},\hat{\nu}\in P_{k}^{+}$ and denote
by $\lambda ,\mu ,\nu \in \mathfrak{P}_{\leq n-1,k}$ the images under the
bijection (\ref{weight2part}). Then one has the following `inverse relation'
to (\ref{curious}),
\begin{equation}
\mathcal{N}_{\lambda \mu }^{\nu }=C_{\lambda \mu }^{\limfunc{Rot}^{-\hat{d}%
}(\nu ),d}\;,\qquad \hat{d}=\frac{|\lambda |+|\mu |-|\nu |}{n},
\label{invcurious}
\end{equation}%
where $\limfunc{Rot}:\mathfrak{P}_{\leq n,k}\rightarrow \mathfrak{P}_{\leq
n,k}$ is the $\widehat{\mathfrak{sl}}(n+k)$-Dynkin diagram automorphism of
order $n+k$ and%
\begin{equation}
d=n-\limfunc{n}\nolimits_{N-\hat{d}}(\nu )=\limfunc{n}\nolimits_{\hat{d}%
+1}(\nu ^{\vee })\;.  \label{dinv}
\end{equation}%
In particular, adopting the same conventions as in Theorem \ref{main} we can
write the fusion coefficient as the following `vacuum expectation value' of
the Clifford algebra,%
\begin{equation}
\mathcal{N}_{\lambda \mu }^{\nu }=\sum_{T}(-1)^{d(n+1)}\langle \varnothing
,\psi _{\ell _{n}(\nu )+\hat{d}}\cdots \psi _{\ell _{1}(\nu )+\hat{d}}\psi
_{\ell _{1}(\mu )+t_{n}}^{\ast }\cdots \psi _{\ell _{n}(\mu )+t_{1}}^{\ast
}\varnothing \rangle ,  \label{fermifusion}
\end{equation}%
where all indices are understood modulo $N$.
\end{corollary}

\begin{proof}
Let $\sigma \in \mathfrak{P}_{\leq n,k}$. First we observe that the column
reduction $\sigma \rightarrow \sigma ^{\prime }$ can be expressed in terms
of the Dynkin diagram automorphism $\limfunc{Rot}$ of order $N$, $\sigma
^{\prime }=\limfunc{Rot}\nolimits^{\ell _{n}(\sigma )-1}\sigma $. Each
application of the automorphism $\func{rot}$ of order $n$ corresponds to
applying $\limfunc{Rot}$ followed by a column reduction. Hence, for any $%
a=1,...,n-1$ we have%
\begin{equation}
\func{rot}^{a}(\sigma ^{\prime })=\limfunc{Rot}\nolimits^{\ell _{n-a}(\sigma
)-1}(\sigma )
\end{equation}%
Setting $a=d=\frac{|\lambda |+|\mu |-|\sigma |}{N}$ and $\hat{d}:=\ell
_{n-d}(\sigma )-1$ we find that%
\begin{equation}
\nu =\func{rot}^{d}(\sigma ^{\prime })=\limfunc{Rot}\nolimits^{\ell
_{n-d}(\sigma )-1}(\sigma )
\end{equation}%
and, hence,%
\begin{equation*}
C_{\lambda \mu }^{\sigma ,d}=\mathcal{N}_{\lambda\mu}^{\func{%
rot}^{d}(\sigma^{\prime })}=\mathcal{N}_{\lambda\mu}^{%
\nu}=C_{\lambda \mu }^{\limfunc{Rot}^{-\hat{d}}(\nu ),d}\;.
\end{equation*}%
To derive the equality (\ref{dinv}) for the degree $\hat{d}$ observe that%
\begin{eqnarray*}
|\lambda |+|\mu |-|\nu | &=&|\lambda |+|\mu |-|\limfunc{Rot}\nolimits^{\ell
_{n-d}(\sigma )-1}(\sigma )| \\
&=&|\lambda |+|\mu |-|\sigma |-N\limfunc{n}\nolimits_{\hat{d}}(\sigma )+n%
\hat{d} \\
&=&|\lambda |+|\mu |-|\sigma |-Nd+n\hat{d} \\
&=&n\hat{d},
\end{eqnarray*}%
where we have used the generally valid formula $|\limfunc{Rot}%
\nolimits^{a}(\sigma )|=|\sigma |+N\limfunc{n}\nolimits_{a}(\sigma )-na$ and
the trivial observation that $\limfunc{n}\nolimits_{\ell _{n-d}(\sigma
)-1}(\sigma )=d$ which is immediate from the definition of $\limfunc{n}%
\nolimits_{a}(\sigma )$.
\end{proof}

\subsection{Recursion formulae for fusion coefficients} The inductive
algorithm presented in \cite[Section 11]{ckcs} for successively computing product
expansions in the ring $qH^{\ast }(\limfunc{Gr}_{n,N})$ from those in $qH^{\ast }(%
\limfunc{Gr}_{n\mp 1,N})$ can be projected onto the respective fusion
rings $V_{k}(\widehat{\mathfrak{sl}}(n);\mathbb{Z})$ and $V_{k\pm 1}(%
\widehat{\mathfrak{sl}}(n\mp 1);\mathbb{Z})$. In particular, one obtains
recursion formulae expressing $\mathcal{N}_{\lambda \mu }^{\nu }(n,k)$ as a sum of the fusion coefficients  $\mathcal{N}_{\tilde{\lambda}\tilde{\mu}}^{\tilde{\nu},\tilde{d}%
}(n\mp 1,k\pm 1)$.
\begin{corollary}
Given $\lambda ,\mu ,\nu
\in \mathfrak{P}_{\leq n-1,k}$ define $d$, $\hat{d}$ as in Corollary \ref%
{lift} and set $\sigma =\limfunc{Rot}^{-\hat{d}}\nu $. Then one has the
relations
\begin{eqnarray*}
\mathcal{N}_{\lambda \mu }^{\nu }(n,k)&=&\sum_{r=0}^{\lambda^t _{1}}(-1)^{d+r+%
\limfunc{n}_{j-1}(\mu )+\limfunc{n}_{j-r-1}(\sigma )}\sum_{\lambda /\rho
=(1^r)}\mathcal{N}_{\rho ^{\prime }(\psi^\ast _{j}\mu )^{\prime }}^{\limfunc{rot}%
^{d_{r}}(\psi^\ast _{j-r}\sigma )^{\prime }}(n+1,k-1)\;,\\
\mathcal{N}_{\lambda \mu }^{\nu }(n,k)&=&\sum_{r=0}^{\lambda _{1}}(-1)^{d+%
\limfunc{n}_{j-1}(\mu )+\limfunc{n}_{j+r-1}(\sigma )}\sum_{\lambda /\rho
=(r)}\mathcal{N}_{\rho ^{\prime }(\psi _{j}\mu )^{\prime }}^{\limfunc{rot}%
^{d'_{r}}(\psi _{j+r}\sigma )^{\prime }}(n-1,k+1)\;,
\end{eqnarray*}%
where $d_r=d$ if $j<r$ and $d_r=d-1$ else. Similarly, $d'_r=d$ if $j+r\leq N$ and $d'_r=d-1$ else.
\end{corollary}
\begin{proof}
Exploiting (\ref{invcurious}) we rewrite the fusion coefficient as Gromov-Witten invariant. Employing the recursion relations for Gromov-Witten invariants \cite[Corollary 11.7]{ckcs} and subsequent application of (\ref{curious}) proves the assertion.
\end{proof}
Successive application of the second recursion formula allows one to reduce the computation of $\widehat{\mathfrak{sl}}(n)_{k}$-fusion coefficients to the computation
of $\widehat{\mathfrak{sl}}(2)_{k+n-2}$-fusion coefficients which are
explicitly known (see e.g. \cite{CFTbook}),%
\begin{equation}
n=2:\;\mathcal{N}_{(\lambda) (\mu) }^{(\nu) }(2,k)=\left\{
\begin{array}{cc}
1, & |\lambda -\mu |\leq \nu \leq \min \{\lambda +\mu ,2k-\lambda -\mu \} \\
0, & \text{else}%
\end{array}%
\right. \;.
\end{equation}

\begin{example}{\rm
We consider the previous example with $n=3,~k=4,~\lambda =(3,1),~\mu =(3,2)$
and $\nu =(2,1)$. Since $w(\mu )=1001010$ only the diagrams $\psi _{j}\mu $
with $j=1,4,6$ are nonvanishing. Let us choose $j=4$ with $\psi _{4}\mu =-(4)
$. Then converting $\nu $ into a 01-word, $w(\nu )=1010100$, we find that acting with $\psi _{j+r}$ on $\limfunc{Rot}^{-%
\hat{d}}\nu =(4,3,2)$ yields a only a nonzero result for the values $r=1,3$. One then calculates,%
\begin{multline*}
\mathcal{N}_{\lambda \mu }^{\nu }(3,4)=C_{\lambda \mu }^{(4,3,2)}(3,4) \\
=(-1)^{d+\limfunc{n}_{3}(\mu )+\limfunc{n}_{4}(4,3,2)}\{C_{(2,1),\psi
_{4}(\mu )}^{\psi _{5}(4,3,2)}(2,5)+C_{(3,0),\psi _{4}(\mu )}^{\psi
_{5}(4,3,2)}(2,5)\}\\+(-1)^{d+\limfunc{n}_{3}(\mu )+\limfunc{n}%
_{6}(4,3,2)}C_{(1,0),\psi _{4}(\mu )}^{\psi _{7}(4,3,2)}(2,5) \\
=C_{(2,1),(4,0)}^{(5,2)}(2,5)+C_{(3,0),(4,0)}^{(5,2)}(2,5)-C_{(1,0),(4,0)}^{(3,2)}(2,5)
\\
=\mathcal{N}_{(1)(4)}^{(3)}(2,5)+\mathcal{N}_{(3)(4)}^{(3)}(2,5)-\mathcal{N}%
_{(1)(4)}^{(1)}(2,5)=1+1-0=2,
\end{multline*}%
which is in accordance with our previous result from Example \ref{rimhookex2}
exploiting (\ref{invcurious}).
}
\end{example}

\subsection{Racah-Speiser algorithms for fusion coefficients} We are now projecting the quantum Racah-Speiser algorithm from the quantum cohomology ring onto the fusion ring and then compare with various other algorithms for computing fusion coefficients.
\begin{corollary}
Given $\lambda ,\mu ,\nu
\in \mathfrak{P}_{\leq n-1,k}$ define $d$, $\hat{d}$ as in Corollary \ref%
{lift}. For a given
permutation $\pi \in S_{n}$ introduce the weight vector
\begin{equation*}
\hat{\alpha}_{i}(\pi )=(\ell _{i}(\nu )-\ell _{\pi (i)}(\mu )+\hat{d})\func{%
mod}N\geq 0
\end{equation*}%
and
\begin{equation*}
d(\pi )=\#\{i~|~(\ell _{i}(\nu )+\hat{d})\func{mod}N<\ell _{\pi (i)}(\mu
)\}\;.
\end{equation*}%
Then one has the following combinatorial expression for the fusion
coefficients,%
\begin{equation}
\mathcal{N}_{\lambda \mu }^{\nu }=\sum_{\substack{ \pi \in S_{n}  \\ d(\pi
)=d }}(-1)^{\ell (\pi )+(n+1)d}K_{\lambda ,\hat{\alpha}(\pi )}\;.
\label{fusionKostka}
\end{equation}
\end{corollary}
\begin{proof}
The assertion immediately follows from \eqref{invcurious} and \eqref{GWKostka}.
\end{proof}
\subsection{Comparison with the Kac-Walton formula}
The Kac-Walton formula \cite{Kac}, \cite{Walton}, \cite{GoodmanWenzl} can be stated as follows: denote by $\hat{W}$ the affine Weyl group, then%
\begin{equation}
\mathcal{N}_{\lambda \mu }^{\nu }=\sum_{\substack{ w\in \hat{W}  \\ w\cdot
\hat{\nu}\in P_{k}^{+}}}(-1)^{\ell (w)}c_{\lambda \mu }^{\sigma (w\cdot \hat{%
\nu})}~,  \label{KWformula}
\end{equation}%
where $w\cdot \hat{\nu}=w(\hat{\nu}+\hat{\rho})-\hat{\rho}$ is the shifted
Weyl group action with $\hat{\rho}=\sum_i\hat\omega_i$ being the affine Weyl vector and $\sigma
(w\cdot \hat{\nu})$ is the partition obtained by adding $m_{0}(\hat{\nu})$
(the zeroth Dynkin label) $n$-columns to the Young diagram of the image of $%
w\cdot \hat{\nu}$ under the bijection (\ref{weight2part}).

\begin{example}
{\rm Setting once more $n=3$ and $k=4$ consider the affine weights $\hat{%
\lambda}=\hat{\omega}_{0}+2\hat{\omega}_{1}+\hat{\omega}_{2}$, $\hat{\mu}=%
\hat{\omega}_{0}+\hat{\omega}_{1}+2\hat{\omega}_{2}$ in $P_{k}^{+}$. The
corresponding partitions under \eqref{weight2part} are $\lambda =(3,1)$ and $\mu =(3,2)$. We already
stated the result of the Littlewood-Richardson rule in Example \ref%
{rimhookex2}, see (\ref{LRRex}). Removing all $n=3$-columns from the
partitions $\nu $ we obtain the $\mathfrak{sl}(n)$ tensor product
decomposition%
\begin{equation*}
\lambda \otimes \mu =(6,3)\oplus (5,1)\oplus (5,4)\oplus 2(4,2)\oplus
(3,0)\oplus (3,3)\oplus 2(2,1)\oplus (0,0)\;.
\end{equation*}%
Here we have identified by abuse of notation highest weight modules with the
corresponding partitions. We wish to consider the fusion coefficient of the
affine weight $\hat{\nu}=2\hat{\omega}_{1}+2\hat{\omega}_{2}$ with partition
$\nu =(4,2)$. From (\ref{KWformula}) we then find%
\begin{equation*}
\mathcal{N}_{\lambda \mu }^{\nu }=c_{\lambda \mu }^{(4,2)}-c_{\lambda \mu
}^{(6,3)}=2-1=1,
\end{equation*}%
since $s_{0}\cdot \widehat{(6,3)}=s_{0}\cdot (-2\hat{\omega}_{0}+3\hat{\omega%
}_{1}+3\hat{\omega}_{2})=\hat{\nu}$ with $s_{0}$ denoting the affine Weyl
reflection. In fact, the entire fusion product expansion is computed to%
\begin{equation}\label{fusionex3}
\Yvcentermath1\yng(3,1)~\ast ~\yng(3,2)= \\
\Yvcentermath1\yng(4,2)+~\yng(3)+~\yng(3,3)+2~\yng(2,1)+\emptyset \;.
\end{equation}%
In order to compare this result with the formulae (\ref{GWKostka}) we
compare with the corresponding product in the quantum cohomology ring of
Example \ref{rimhookex2}. From which we read off %
\begin{equation*}
\mathcal{N}_{\lambda \mu }^{\nu }=C_{\lambda \mu }^{\limfunc{Rot}^{-1}\nu
,1}=C_{\lambda \mu }^{(2,0,0),1}=K_{\lambda ,(2,1,1)}-K_{\lambda
,(3,1,0)}=c_{\lambda (2,1)}^{(4,2,1)}-c_{\lambda (1,0,0)}^{(4,1,0)}=2-1=1\;.
\end{equation*}%
The rim-hook algorithm $\mathrm{(\ref{GWrimhook})}$ yields%
\begin{equation*}
\mathcal{N}_{\lambda \mu }^{\nu }=C_{\lambda \mu }^{\limfunc{Rot}^{-1}(\nu
),1}=c_{(3,1,0,0),(3,2,0,0)}^{(4,3,1,1)}=1\;.
\end{equation*}
}
\end{example}

\begin{remark}{\rm Note that as pointed out in the introduction the projected `quantum Racah-Speiser algorithm' does not use the affine Weyl group but only the symmetric group. Moreover, the summands in the expression for the fusion coefficients obtained via the
reduction formula (\ref{curious}) from either (\ref{GWKostka}) or (\ref%
{GWrimhook}) in general differ from the Kac-Walton formula (\ref{KWformula}%
), thus, leading to non-trivial identities between Littlewood-Richardson
coefficients.
}
\end{remark}

Both isomorphisms \eqref{VQiso} and \eqref{VQisodual} of Theorem \ref{quotient} give rise to algorithms for the computation of fusion coefficients. The first isomorphism \eqref{VQiso} leads simply to the projection of the dual rim hook algorithm for the quantum cohomology ring discussed in Section \ref{subsec:dualrimhook} and Example \ref{rimhookex2}.
\subsection{Projection of the dual rim hook algorithm}
\begin{enumerate}
\item Compute via the Littlewood-Richardson algorithm the expansion $%
s_{\lambda }s_{\mu }=\sum_{\rho }c_{\lambda \mu }^{\rho }s_{\rho }$. Discard
all terms for which the partition $\rho $ has length \TEXTsymbol{>} $n$.

\item Replace each Schur polynomial $s_{\rho }$ with $s_{\rho ^{\prime }}$,
i.e. remove all $n$-columns in the Young diagram associated with $\rho $.
(This is allowed since $e_{n}=1$ according to \eqref{VQiso}.)

\item Among the set of remaining Schur polynomials make for each $\rho
^{\prime }$ with $\rho _{1}^{\prime }>k$ the replacement $s_{\rho ^{\prime
}}=(-1)^{d(n-1)}h_{k}^{d}s_{v(\rho ^{\prime })}$, where $v(\rho ^{\prime })$
is the integer vector defined in \eqref{GWreduce} and $|\rho ^{\prime }|-|v(\rho ^{\prime
})|=Nd$. Use the straightening rules \eqref{Schurstraight} for Schur polynomials to express $%
s_{v(\rho ^{\prime })}$ in terms of a Schur polynomial $s_{\sigma }$ with $%
\sigma $ being a \emph{partition}. Then collect terms noting that $%
h_{k}^{d}s_{\sigma }=s_{\func{rot}^{d}(\sigma )}$ according to the
Pieri-rule for Schur polynomials. The resulting coefficients are the fusion coefficients $\mathcal{N}_{\lambda \mu }^{\nu }$ with $\nu =\func{rot}^{d}(\sigma )$.
\end{enumerate}
Since the isomorphism \eqref{VQiso} has not been explicitly
stated previously, we briefly outline its derivation which closely
parallels the one for the second isomorphism \eqref{VQisodual} which is described in detail in \cite[Proof of Theorem 6.20]{ckcs}. The above algorithm then follows along the same lines as in the case of the dual rim hook algorithm discussed previously.
\begin{proof}[Proof of the isomorphism \eqref{VQiso} and derivation of the algorithm]
Consider the complexification $\Lambda_\mathbb{C}^{(n)}=\mathbb{C}\otimes_\mathbb{Z}\Lambda^{(n)}$ of $\Lambda^{(n)}=\mathbb{Z}[e_1,\ldots,e_n]$. Let $\mathcal{J}$ be the ideal in \eqref{VQiso}. Along the same lines as in \cite{ckcs} one shows that the ideal $\mathcal{J}$ is radical. Again let $\mathbb{V}(\mathcal{J})$ be the set of $n$-tuples $y=(y_1,\ldots,y_n)$ for which $f\in \mathcal{J}$ vanishes. The latter is identical with the solutions of the following system of equations%
\begin{equation}\label{dualBAE}
y_{1}^{n+k}=\cdots =y_{n}^{n+k}=(-1)^{n-1}h_{k}(y_{1},\ldots ,y_{n})\;.
\end{equation}%
Using the same arguments as in \cite[Theorem 6.4]{ckcs} one shows that there exists a bijection between the set of partitions $\mathfrak{P}_{\leq n-1,k}$ and the set of solutions (up to permutation of the $y_i$'s),
\begin{equation}\label{dualBAEsol}
\mathfrak{P}_{\leq n-1,k}\ni\sigma\mapsto y_\sigma=\zeta^{\frac{|\sigma|}{n}}(\zeta^{I_1(\sigma)},\ldots,\zeta^{I_n(\sigma)}),
\end{equation}
where $\zeta=\exp(\frac{2\pi i}{k+n})$ and $I=I(\sigma)$ is the tuple of half-integers defined previously in \eqref{Imap}. According to Hilbert's Nullstellensatz it then follows that two elements $f,g\in\Lambda_\mathbb{C}^{(n)}/\mathcal{J}$ are identical if and only if they coincide on the solution set \eqref{dualBAEsol}. This result together with Lemma \ref{technical} implies the desired isomorphism, since one can now send $\hat\lambda\in P^+_k$ to $s_\lambda(y_\sigma)=\zeta^{|\sigma||\lambda|}s_{\lambda^t}(\zeta^{-I(\sigma^t)})$ and then apply the second isomorphism \eqref{VQisodual}; see \cite[Proof of Theorem 6.20]{ckcs}.

In order to derive the algorithm we now exploit \eqref{dualBAE} and proceed analogously to the derivation of the dual rim hook algorithm detailed in Section \ref{subsec:dualrimhook}.
\end{proof}
\begin{remark}{\rm
As we have not made use of \eqref{curious} in the derivation of the algorithm, the above line of argument provides an alternative proof of Proposition \ref{mainprop}.}
\end{remark}
\subsection{A dual Racah-Speiser algorithm for the fusion ring}
The second isomorphism \eqref{VQisodual} in Theorem \ref{quotient} provides a dual algorithm in terms of the transposed partitions:

\begin{enumerate}
\item Compute via the Littlewood-Richardson algorithm the expansion $%
s_{\lambda ^{t}}s_{\mu ^{t}}=\sum_{\rho ^{t}}c_{\lambda \mu }^{\rho }s_{\rho
^{t}}$; note that $c_{\lambda \mu }^{\rho }=c_{\lambda ^{t}\mu ^{t}}^{\rho
^{t}}$. Discard all terms for which the partition $\rho ^{t}$ has length
\TEXTsymbol{>} $k$.

\item For each of the remaining terms with $\rho _{1}^{t}>n$ make the
replacement $s_{\rho ^{t}}=(-1)^{d(k-1)}e_k^d s_{v(\rho^t)}$ with $dN=|\rho|-|v(\rho^t)|$ and
\begin{equation*}
v_i(\rho^t)=\rho^t_i,\qquad i-k\leq v_i(\rho^t)<i+n\,.
\end{equation*}
Then use the straightening rules \eqref{Schurstraight} for Schur
polynomials to rewrite $s_{v(\rho^t)}$ as $s_{\sigma^{t}}$ with $\sigma^{t}$ a partition. Finally, observe that $e_k^d s_{\sigma^t}=s_{\nu^t}$, where $\nu^t$ is obtained by adding $d$ $k$-columns to $\sigma^t$.
\item Remove all rows of length $n$ in $\nu^{t}$ to obtain $(\nu^t)''=(\nu')^t$; this is allowed because $h_n=1$. Collecting terms one
obtains the fusion coefficient $\mathcal{N}_{\lambda \mu }^{\nu' }$.
\end{enumerate}
\begin{proof}[Derivation of the algorithm]
The proof is completely analogous to the one considered above, hence we omit the details. The only important information needed is that the zero set of the ideal in \eqref{VQisodual} is given by the solutions to the Bethe Ansatz equations of the phase model (see \cite[Proposition 6.1 and Lemma 6.3]{ckcs}),
\begin{equation}\label{phaseBAE}
x_1^{n+k}=\cdots =x_k^{n+k}=(-1)^{k-1}e_k(x_1,\ldots,x_k)\,.
\end{equation}
Using the same presentation of the Schur polynomial as in \eqref{Schurpolyrep} one successively replaces powers $>n$ in the variable $x_i$ via the Bethe Ansatz equations \eqref{phaseBAE} to arrive at the above replacement rule.
\end{proof}
\begin{example}
Exploiting that $c_{\lambda \mu }^{\rho }=c_{\lambda ^{t}\mu ^{t}}^{\rho
^{t}}$ we consider once more the Littlewood-Richardson expansion from
Example \ref{rimhookex2} with $n=3,\;k=4,\;\lambda ^{t}=(2,1,1,0)$ and $\mu
^{t}=(2,2,1,0)$. After taking the transpose partitions we discard $\rho
^{t}=(2,2,2,1,1,1)$, $(3,2,1,1,1,1)$, $(2,2,2,2,1)$, $(3,2,2,1,1)$, $(3,3,1,1,1)$ and $(4,2,1,1,1)$ from (\ref{LRRex}). We are left with three partitions $\rho ^{t}
$ for which $\rho _{1}^{t}>n$, namely $(4,2,2,1)$, $(4,3,1,1)$, $(4,3,2,0)$.
Employing the above algorithm we calculate%
\begin{equation*}
s_{(4,2,2,1)}=s_{(2,2,1,1)},\;s_{(4,3,1,1)}=s_{(3,1,1,1)},%
\;s_{(4,3,2,0)}=s_{(3,2,0,1)}=0\;.
\end{equation*}%
Removing all rows of length $n=3$ and collecting terms one finds after taking the transpose partitions the previous expansion \eqref{fusionex3}.%
\begin{equation*}
\end{equation*}
\end{example}

%\bibliographystyle{plain}
%\bibliography{refsqcohom}

\begin{thebibliography}{10}

\bibitem{Beauville}
A.~Beauville.
\newblock Conformal blocks, fusion rules and the {V}erlinde formula.
\newblock In {\em Proceedings of the {H}irzebruch 65 {C}onference on
  {A}lgebraic {G}eometry ({R}amat {G}an, 1993)}, volume~9 of {\em Israel Math.
  Conf. Proc.}, pages 75--96, 1996.

\bibitem{Bertram}
A.~Bertram.
\newblock Quantum {S}chubert calculus.
\newblock {\em Adv. Math.}, 128(2):289--305, 1997.

\bibitem{BCF}
A.~Bertram, I.~Ciocan-Fontanine, and W.~Fulton.
\newblock Quantum multiplication of {S}chur polynomials.
\newblock {\em J. Algebra}, 219(2):728--746, 1999.

\bibitem{Buch}
A.~S. Buch.
\newblock Quantum cohomology of {G}rassmannians.
\newblock {\em Compositio Math.}, 137(2):227--235, 2003.

\bibitem{CFTbook}
P.~Di~Francesco, P.~Mathieu, and D.~S{\'e}n{\'e}chal.
\newblock {\em Conformal field theory}.
\newblock Graduate Texts in Contemporary Physics. Springer-Verlag, 1997.

\bibitem{FultonYT}
W.~Fulton.
\newblock {\em Young tableaux}, volume~35 of {\em London Mathematical Society
  Student Texts}.
\newblock Cambridge University Press, 1997.
\newblock With applications to representation theory and geometry.

\bibitem{FuHa}
W.~Fulton and J.~Harris.
\newblock {\em Representation theory}, volume 129 of {\em Graduate Texts in
  Mathematics}.
\newblock Springer-Verlag, New York, 1991.
\newblock A first course, Readings in Mathematics.

\bibitem{GoodmanWenzl}
F.~M. Goodman and H.~Wenzl.
\newblock Littlewood-{R}ichardson coefficients for {H}ecke algebras at roots of
  unity.
\newblock {\em Adv. Math.}, 82(2):244--265, 1990.

\bibitem{Kac}
V.~G. Kac.
\newblock {\em Infinite-dimensional {L}ie algebras}.
\newblock Cambridge University Press, second edition, 1985.

\bibitem{KacPeterson}
V.~G. Kac and D.~H. Peterson.
\newblock Infinite-dimensional {L}ie algebras, theta functions and modular
  forms.
\newblock {\em Adv. in Math.}, 53(2):125--264, 1984.

\bibitem{Kingetal}
R.~C. King, C.~Tollu, and F.~Toumazet.
\newblock Stretched {L}ittlewood-{R}ichardson and {K}ostka coefficients.
\newblock In {\em Symmetry in physics}, volume~34 of {\em CRM Proc. Lecture
  Notes}, pages 99--112. Amer. Math. Soc., Providence, RI, 2004.

\bibitem{ckcs}
C.~Korff and C.~Stroppel.
\newblock The $\widehat{\mathfrak{sl}}(n)_{k}$-{WZNW} fusion ring: a
  combinatorial construction and a realisation as quotient of quantum
  cohomology.
\newblock {\em Newton Institute Preprint}, NI09064-DIS/ALT, 2009.
\newblock{http://arxiv.org/abs/0909.2347}

\bibitem{MacDonald}
I.~G. Macdonald.
\newblock {\em Symmetric functions and {H}all polynomials}.
\newblock Oxford Mathematical Monographs. The Clarendon Press Oxford University
  Press, second edition, 1995.

\bibitem{Postnikov}
A.~Postnikov.
\newblock Affine approach to quantum {S}chubert calculus.
\newblock {\em Duke Math. J.}, 128(3):473--509, 2005.

\bibitem{Racah}
G.~Racah.
\newblock Lectures on {L}ie groups.
\newblock In {\em Group theoretical concepts and methods in elementary particle
  physics ({L}ectures {I}stanbul {S}ummer {S}chool {T}heoret. {P}hys., 1962)},
  pages 1--36. Gordon and Breach, New York, 1964.

\bibitem{RRW}
M.~S. Ravi, J.~Rosenthal, and X.~Wang.
\newblock Degree of the generalized {P}l\"ucker embedding of a {Q}uot scheme
  and quantum cohomology.
\newblock {\em Math. Ann.}, 311(1):11--26, 1998.

\bibitem{Rietsch}
K.~Rietsch.
\newblock Quantum cohomology rings of {G}rassmannians and total positivity.
\newblock {\em Duke Math. J.}, 110(3):523--553, 2001.

\bibitem{ST}
B.~Siebert and G.~Tian.
\newblock On quantum cohomology rings of {F}ano manifolds and a formula of
  {V}afa and {I}ntriligator.
\newblock {\em Asian J. Math.}, 1(4):679--695, 1997.

\bibitem{Sottile}
F.~Sottile.
\newblock Rational curves on {G}rassmannians: systems theory, reality, and
  transversality.
\newblock In {\em Advances in algebraic geometry motivated by physics
  ({L}owell, {MA}, 2000)}, volume 276 of {\em Contemp. Math.}, pages 9--42.
  Amer. Math. Soc., Providence, RI, 2001.

\bibitem{Speiser}
D.~Speiser.
\newblock Theory of compact {L}ie groups and some applications to elementary
  particle physics.
\newblock In {\em Group theoretical concepts and methods in elementary particle
  physics ({L}ectures {I}stanbul {S}ummer {S}chool {T}heoret. {P}hys., 1962)},
  pages 201--276. Gordon and Breach, New York, 1964.

\bibitem{Tamvakis}
H.~Tamvakis.
\newblock Gromov-{W}itten invariants and quantum cohomology of {G}rassmannians.
\newblock In {\em Topics in cohomological studies of algebraic varieties},
  Trends Math., pages 271--297. Birkh\"auser, Basel, 2005.

\bibitem{Verlinde}
E.~Verlinde.
\newblock Fusion rules and modular transformations in {$2$}{D} conformal field
  theory.
\newblock {\em Nuclear Phys. B}, 300(3), 1988.

\bibitem{Walton}
Mark~A. Walton.
\newblock Fusion rules in {W}ess-{Z}umino-{W}itten models.
\newblock {\em Nuclear Phys. B}, 340(2-3):777--790, 1990.

\bibitem{Witten}
E.~Witten.
\newblock The {V}erlinde algebra and the cohomology of the {G}rassmannian.
\newblock In {\em Geometry, topology, \& physics}, Conf. Proc. Lecture Notes
  Geom. Topology, IV, pages 357--422. Int. Press, Cambridge, MA, 1995.

\end{thebibliography}

\end{document}